\documentclass[12pt,leqno]{amsart}
\usepackage[dvips]{graphics}
\usepackage{amssymb}
\usepackage{color}

\numberwithin{equation}{section}

\newtheorem{thm}{Theorem}[section]
\newcommand{\bt}{\begin{thm}}
\newcommand{\et}{\end{thm}}

\newtheorem{cor}[thm]{Corollary}
\newcommand{\bc}{\begin{cor}}
\newcommand{\ec}{\end{cor}}

\newtheorem{lem}[thm]{Lemma}
\newcommand{\bl}{\begin{lem}}
\newcommand{\el}{\end{lem}}

\newtheorem{prop}[thm]{Proposition}
\newcommand{\bp}{\begin{prop}}
\newcommand{\ep}{\end{prop}}

\newtheorem{defn}[thm]{Definition}
\newcommand{\bd}{\begin{defn}}
\newcommand{\ed}{\end{defn}}

\newtheorem{rmrk}[thm]{Remark}
\newcommand{\br}{\begin{rmrk}}
\newcommand{\er}{\end{rmrk}}

\newtheorem{quest}[thm]{Question}
\newcommand{\bq}{\begin{quest}}
\newcommand{\eq}{\end{quest}}

\newtheorem{example}[thm]{Example}

\newdimen\vintkern\vintkern12pt
\def\vint{-\kern-\vintkern\int}

\newcommand{\hm}{{\mathcal H}}

\newcommand{\diam}{\operatorname{diam}}
\newcommand{\trace}{\operatorname{tr}}
\newcommand{\length}{\ell}
\newcommand{\Area}{\operatorname{Area}}

\newcommand{\md}{\operatorname{md}}

\newcommand{\Fill}{{\operatorname{Fill}}}

\newcommand{\jac}{{\mathbf J}}
\newcommand{\ap}{\operatorname{ap}}
\newcommand{\apmd}{\ap\md}

\newcommand{\tref}[1]{Theorem~\ref{#1}}
\newcommand{\cref}[1]{Corollary~\ref{#1}}
\newcommand{\pref}[1]{Proposition~\ref{#1}}

\newcommand{\lref}[1]{Lemma~\ref{#1}}

\newcommand{\R}{\mathbb{R}}

\newcommand{\J}{\mathbf{J}}

\begin{document}
\pagebreak


\title{Intrinsic structure of minimal discs in metric spaces}







\author{Alexander Lytchak}

\address
  {Mathematisches Institut\\ Universit\"at K\"oln\\ Weyertal 86 -- 90\\ 50931 K\"oln, Germany}
\email{alytchak@math.uni-koeln.de}

\author{Stefan Wenger}

\address
  {Department of Mathematics\\ University of Fribourg\\ Chemin du Mus\'ee 23\\ 1700 Fribourg, Switzerland}
\email{stefan.wenger@unifr.ch}

\date{\today}

\thanks{The second author was partially supported by Swiss National Science Foundation Grants 153599 and 165848}

\begin{abstract}
 We  study the intrinsic structure  of   parametric minimal discs in metric spaces admitting a quadratic isoperimetric inequality.  We associate to each minimal disc
a compact, geodesic metric space whose geometric, topological, and analytic properties are controlled by the isoperimetric inequality. Its
geometry can be used to control the shapes of all curves and therefore
the geometry and topology of  the original metric space.
 The class of spaces arising in this way as intrinsic minimal discs
is a natural generalization of the class of Ahlfors regular discs, well-studied  in
analysis on metric spaces.
\end{abstract}

\maketitle

\maketitle
\renewcommand{\theequation}{\arabic{section}.\arabic{equation}}
\pagenumbering{arabic}

\section{Introduction}\label{sec:Intro}

\subsection{Motivation}
A  smooth minimal surface in a  Riemannian manifold has vanishing mean curvature.  The Gauss equation then forces the intrinsic  curvature of the minimal surface to be no larger  than that of the ambient space.  This has strong implications on the intrinsic geometry and thus on the local and global shape of the minimal surface.   Given a smooth Jordan curve $\Gamma $ in a Riemannian manifold $M$ of bounded geometry, the  classical solution of the  Plateau problem
(\cite{Dou31},  \cite{Rad30}, \cite{Mor48}) provides a minimal disc
$u:D\to M$ spanned by $\Gamma$. This solution, which is a priori
constructed in a Sobolev space, turns out to be  a smooth map. Moreover, it is an immersion up to finitely many \emph{branch points}.
In the absence of branch points, the minimal disc with its induced Riemannian metric turns out  to be a smooth $2$-dimensional Riemannian manifold $Z$ with boundary given by  $\Gamma$.  The map $u$ factors as  $u=\bar u\circ P$, where  $P:\bar D\to Z$ is a diffeomorphism and $\bar u:Z\to M$  is a Riemannian immersion, thus $\bar u$   preserves the length of all curves.   Moreover, the Gauss equation and other implications of minimality provide restrictions on the geometry of $Z$.

 The aim of the present article is to investigate the intrinsic geometry
of minimal discs in the much broader setting of metric spaces with quadratic isoperimetric inequality  and to
find structures analogous to those observed in the  setting of
Riemannian manifolds. Given any \emph{minimal disc}
$u:\bar D\to X$ in such a metric space $X$
we would like to find  a nice metric space $Z$  whose properties reflect  geometric properties of $X$, and such that $u$ factorizes into a homeomorphism from $\bar D$ to $ Z$ and a length preserving  immersion $Z\to X$.  One cannot achieve this in full, as  examples will demonstrate, but \emph{almost}  as our results will show. The geometric properties of the space $Z$ control the shape of the {Jordan curve $\Gamma$ and therefore the geometry of $X$.  Results obtained in this paper are used in \cite{LW-isop}
to prove an isoperimetric characterization of upper curvature bounds, in \cite{LW15-asymptotic} to study topological and asymptotic properties of  spaces with quadratic isoperimetric inequality  and in \cite{LW-ahlfors} to find natural parametrizations of Ahlfors $2$-regular discs and their generalizations.

 \subsection{Setting and construction}
In \cite{LW}, we have found a solution to the classical Plateau problem in any proper metric space $X$. Given any Jordan curve $\Gamma \subset X$  one would like to find a disc bounded by $\Gamma$ with minimal (parametrized Hausdorff) area. As in  the classical situation,  it is   natural to look for a solution in   the set $\Lambda (\Gamma, X)$ of all Sobolev maps  $v\in W^{1,2}  (D,X)$, whose trace $tr (v)$ is a weakly monotone reparametrization of the curve $\Gamma$.   If $X$ is a proper metric space  and $\Lambda (\Gamma, X)$ is not empty then there indeed exists a map
$u$  with minimal area  $\Area (u)$ in
$\Lambda (\Gamma ,X)$, as we showed in \cite{LW}.
We   found a special area  minimizer $u \in \Lambda (\Gamma ,X)$  which moreover  has  minimal (Reshetnyak) energy $E_+^2(u)$ among all area minimizers in $\Lambda (\Gamma ,X)$.   We will call such a map $u$  a \emph{minimal disc} or \emph{solution of the Plateau problem for $(\Gamma,X)$}.

  If no restrictions are imposed on $X$, then a solution  $u$ of the Plateau problem
  can be as irregular as any Sobolev map.
   The situation changes under the natural assumption that in $X$  any  Lipschitz curve
$\gamma :S^1\to X$ of length $l \leq l_0$  bounds a (Sobolev) disc
$v : D \to X$
of  area $A\leq Cl^2$, where $C,l_0>0$ are fixed constants. We say that $X$ admits a $(C,l_0)$-isoperimetric inequality (for the Hausdorff area).
 Many geometrically significant spaces like  compact   Lipschitz manifolds,   compact spaces with one-sided curvature bounds,
Banach spaces and many others satisfy this assumption.

Under this isoperimetric assumption the classical "a priori Hoelder estimates" apply, and any solution $u\in W^{1,2} (D,X)$ of the Plateau problem for $(\Gamma ,X)$ turns out to have a  continuously extendible
representative $u:\bar D\to X$.
The continuity of the solution $u$ makes it from an  almost everywhere  defined map  to a geometric and topological  object.
 We borrow the recipe for the construction of the space $Z$  from the smooth situation.
As to any continuous map one can associate to $u$ an \emph{intrinsic pseudo-distance}
$d_u :\bar D\times \bar D\to [0,\infty]$  by
$$d_{u} (z_1,z_2):=\inf \{\text{length of} \; u\circ \gamma \;  |  \; \gamma \; \text{curve in} \; \bar D \; \text{between}  \;  z_1 \; \text{and} \; z_2  \}.$$
It turns out, cf. \pref{prop:first} below, that $d_u$ is finite-valued. Hence, identifying points on $\bar D$ with $d_u$-distance $0$ from each other we obtain a metric space $Z=Z_u$ which we will call the
\emph{intrinsic minimal disc associated with $u$}. Since $Z_u$ arises from $\bar D$ by an identification of some points we have a canonical projection $P:\bar D\to Z$.



\subsection{Properties of the intrinsic minimal disc} \label{subsec:second}    We fix the following setting for the whole subsection.

{\bf General setting:}
\begin{itemize}
 \item $X$ is a complete metric space admitting a
$(C,l_0)$-isoperimetric inequality;
 \item $\Gamma\subset X$ is a Jordan curve of finite length;
 \item $u: \bar{D}\to X$ is a solution of the Plateau problem for
$(\Gamma, X)$;
\item $d_u:\bar D\times \bar D \to [0,\infty)$ is the pseudo-distance induced by $u$;
\item $Z=Z_u$ is the intrinsic minimal disc associated with $u$;
\item $P:\bar D\to Z$ is the induced canonical projection;
\end{itemize}
Morrey's proof of the Hoelder regularity of minimal discs, generalized in \cite{LW} to metric spaces leads to:
\begin{thm} \label{prop:first}
In the general setting and notations introduced above, the pseudo-distance $d_u$ assumes only finite values  and is continuous. The metric space  $Z=Z_u$  obtained from the pseudo-metric $d_u$
is compact and geodesic. The canonical projection $P:\bar D\to Z$ is continuous.
 The map $u:\bar D\to X $ has a canonical factorization  $u=\bar u\circ P$, where
$\bar u:Z\to X$ is a $1$-Lipschitz map. For any curve $\gamma$ in $\bar D$ the lengths
of $P\circ \gamma$ and $u\circ \gamma$ coincide, thus $\bar u$ preserves the length of $P\circ \gamma$.
\end{thm}

We  are going to  discuss   the topological, geometric, and analytic properties of the constructed space $Z$.

\begin{thm} \label{thmA}
The  intrinsic minimal disc  $Z$ is  homeomorphic to $\bar D$.  The Hausdorff area $\mathcal H^2 (Z)$ and the
length $\ell (\partial Z)$ of the boundary circle are finite.
The domain $\Omega \subset Z$ enclosed by  any Jordan curve  of length $l<l_0$ in $Z$ satisfies
    \begin{equation} \label{Z-hausdorff-isop}
    \mathcal H^2 (\Omega) \leq C \cdot l^2.
\end{equation}
\end{thm}

The isoperimetric property of the topological disc $Z$ has strong implications: lower bound on the area growth, linear local contractibility and
the existence of controlled decompositions into subsets of small diameter.  For the sake of simplicity, we formulate these consequences here only in the case of short boundary curves $\Gamma$, referring to Section \ref{sec:inner-structure-mindiscs}  for the general case.
\begin{cor} \label{cor:neu}
Let $Z$ be as above and assume that  the length $l$  of the boundary circle $\partial Z$ is at most $l_0$. Then
\begin{enumerate}
\item For every $z \in Z$ and any $r\leq d(\partial Z ,z) $ the area of the ball $B(z,r)$ of radius $r$ around $z$  is bounded by
   \begin{equation} \label{Z-hausdorff-growth}
   \mathcal H^2 (B(z,r)) \geq \left(\frac {\pi} {4}\right)^2 \cdot  \frac 1 {4C} \cdot r^2.
   \end{equation}
	\item For  any $z\in Z$ and $r>0$
 the ball $B(z,r)$ in $Z$ is contractible inside  the ball $B(z, (8C+1) \cdot r)$.
\item  There exists a constant
$M=M(C)>0$ such that the  following holds true. For  every $n>0$ there exists a finite, connected   graph $\partial Z\subset G_n\subset Z$
such that any component of $Z\setminus G_n$ is a disc of diameter at most $\frac {l } n$  and such that the number of these components is at most
$M\cdot n^2$.
\end{enumerate}
\end{cor}

\tref{thmA} and \cref{cor:neu}  may be considered as    weak analogues of the statement
 that the curvature of a minimal surface is not larger than that of the ambient space. Indeed, on a smooth surface the isoperimetric inequality  is closely related to upper bounds on the curvature and lower bounds on the injectivity radius, a circumstance which will be analyzed in depth in \cite{LW-isop}.
Similarly, lower bounds on the volume of balls are well-known to be related  to the curvature bounds of the manifold, for instance, by the Theorem of   Bishop--Gromov. We will deduce in \cite{LW-isop} from  \cref{cor:neu} that
 $Z$ inherits from $X$  upper curvature bounds in the sense of Alexandrov, cf.  \cite{Petrunin-old}, \cite{Petrunin}

\cref{cor:neu} shows that $Z$ is metrically very similar to the Euclidean  disc.
According to (i) and (ii) of \cref{cor:neu} and  \cite{Bonk-Kleiner}, \cite{Wildrick}, the space
 $Z$ is locally \emph{quasi-symmetric} to the unit disc if and only if areas of  balls in  $Z$
have a quadratic \emph{upper} bound in terms of the radius.  However, this need not be the case in general. Indeed, $Z$ might arise from $\bar D \subset \R^2$  by collapsing a closed ball
$B\subset \bar D$ to a point,   see Example \ref{disc-collapse}. The decomposition result  (iii) is a topological-geometric version of a similar discrete
statement  in groups with quadratic Dehn function proved by Papasoglu in \cite{Papasoglu}. It immediately implies that the set of isometry classes
of spaces $Z$ arising in \tref{thmA} is a relatively compact set with respect to the Gromov-Hausdorff topology, once $C$ and $l_0$ are  fixed and
$\Gamma$ has length at most $l_0$.

 We emphasize that \cref{cor:neu} follows only from the assumptions that $Z$ is a geodesic $2$-dimensional  disc with the isoperimetric property \eqref{Z-hausdorff-isop}.  This observation is used in \cite{LW-ahlfors}  as
the starting point of further investigations  of all such discs $Z$, also shedding new light on the well-investigated theory of Ahlfors $2$-regular discs,  cf. \cite{Bonk-Kleiner}, \cite{Rajala} and the literature therein.

The topological and analytic properties of the map $P:\bar D\to Z$ are summarized in the next theorem.
\begin{thm} \label{thmB}
The canonical projection $P:\bar D\to Z$
is a uniform limit of homeomorphisms $P_i:\bar D\to Z$. Moreover,
\begin{enumerate}
\item  $P\in \Lambda (\partial Z, Z)
 \subset W^{1,2} (D,Z)$.
\item $P:D\to Z$ is contained  in $W^{1,p} _{loc} (D,Z)$  for some  some $p>2$
depending on $C$.
\item $P:D\to Z$ is locally $\alpha$-Hoelder with $\alpha =\frac \pi 4 \cdot \frac  1 {4\pi  C}$.
\item     The equality $\mathcal H^2 ( P(V)) =\Area (P|_V  ) =\Area (u|_V)$  holds for all
open subsets $V\subset D$.
\end{enumerate}
\end{thm}

The claim  that $P$ is a uniform limit of homeomorphisms
comes  as close as possible to the statement   that $P:\bar D\to Z$ is a diffeomorphism in the classical smooth case.
Even in  the smooth case, if  branch points are present,    the natural map $P:\bar D\to Z$  is not biLipschitz.
 However, in the smooth case,   branch points are isolated and this  forces the map $P$  to be a homeomorphism. In our general setting, the set of  "branch points" does not have to be discrete. Indeed, the map $u$ (and then also $P$) may send an open subset of $\bar D$ to a single point of $X$ (respectively of $Z$),
as in   Example \ref{disc-collapse} already mentioned above.

Inequality  \eqref{Z-hausdorff-growth} and  the constant $\alpha$ in (iii) of \tref{thmB} are optimal at most up to the factor $(\frac \pi 4 )^2$ and
$\frac \pi 4$, respectively, as the example of a cone over a short circle $\Gamma$ shows, see Example \ref{alpha}.
This factor is related to  the  co-area formula in normed planes. It can be replaced  by $1$ if only Euclidean norms appear as metric differentials of $u$,  for instance, in  spaces with upper or lower  curvature bounds, cf. Subsection \ref{subsec:infinite}.  Another possibility to get an optimal factor is to use instead of $\mathcal H^2$ another definition of area, see Subsection \ref{subsec:genarea}  below.

The next theorem describes the map $\bar u:Z\to X$     as an \emph{almost everywhere  infinitesimal isometry}, thus as an "almost Riemannian immersion".  In particular, $\bar u$ preserves
$\mathcal H^2$ up to multiplicities.

\begin{thm}  \label{thmC}
 Let $\bar u:Z\to X$ be the canonical map of the minimal disc $Z$ to $X$ described in \tref{prop:first}.
There exists a decomposition $Z=S\cup _{1\leq i<\infty} K_i$ with compact $K_i$ and
$\mathcal H^2 (S)=0$ such that the restrictions $\bar u:K_i \to \bar u(K_i)$ of the $1$-Lipschitz map $\bar u$ are  biLipschitz. Moreover,  for any $1\leq i <\infty$ and any $x\in K_i$ we have $$\lim _{y\to x, y\in K_i} \frac {d_Z(x,y) } {d_X(\bar u (x) ,\bar u (y))} =1.$$
The map
$\bar u$ sends $\partial Z$ to $\Gamma$ by an arclength preserving homeomorphism.
\end{thm}

If  $\Gamma$ is a \emph{chord-arc  curve}, thus biLipschitz  equivalent  to $S^1$, then one can control the regularity of $u$, $P$  and $Z$ uniformly up to the boundary, as is often the case in the investigation of the Plateau problem:

\begin{thm}  \label{arc-chord}
Assume in addition that $\Gamma$ is a chord-arc curve. Then $P\in W^{1,p} (D,Z)$
 for some $p >2$. In particular, $P:\bar D\to Z$ is globally $(1-\frac 2 p)$-Hoelder continuous.
Moreover, there exists  $\delta >0$
such that   for all
$z\in Z$ and all $0\leq r < \delta$ we have
$\mathcal H^2 (B(z,r)) \geq  \delta  \cdot r ^2$.
\end{thm}

The exponent $p$ and  the  non-collapsing number $\delta $ are bounded in terms of $C,l_0$ and the biLipschitz constant of $\Gamma$.

\subsection{Area minimizers for different areas} \label{subsec:genarea}
 There are several natural ways of measuring the area of $2$-rectifiable subsets  in general metric spaces, beyond the Hausdorff area used in the results above.
Any  choice of a
definition  of area $\mu$ in the sense of convex geometry (see Section \ref{sec:prelim}, \cite{LW-energy-area}, \cite{AlvT04} and the literature therein) provides a natural way to assign the $\mu$-area
 $\Area _{\mu} (u)$ to any Sobolev disc $u:D\to X$.
Among such definitions of area the most important ones are the (Busemann)-Hausdorff area $\mathcal H^2$, the
Holmes-Thompson area $\mu^{ht}$, the
(Benson)-Gromov $m^{\ast}$-measure and (Ivanov's) inscribed Riemannian area $\mu^i$.  Due to \cite{LW}, for any quasi-convex definition of area $\mu$ (for instance, for the four examples above)
 one can find a minimizer of the   $\mu$-area in any non-empty set $\Lambda (\Gamma, X)$, whenever $\Gamma$ is a Jordan curve in a proper metric
space $X$. As in the case of the Hausdorff measure,  we can find  a map $u$ with minimal Reshetnyak energy  $E_+^2(u)$  among all such minimizers. We call such a map $u\in \Lambda (\Gamma, X)$
 a \emph{$\mu$-minimal map}.
 Quasi-convexity of $\mu$ is essential for the existence of $\mu$-minimal maps, but does not play a role in
 the regularity questions discussed in \cite{LW} and here.

If $X$ admits a $(C,l_0)$-isoperimetric inequality for the definition of area  $\mu$ (cf. Section \ref{sec:quadr} and \cite{LW}),
and if $u\in \Lambda (\Gamma ,X)$ is $\mu$-minimal,   then again $u$ has
a  continuous representative  $u:\bar D\to X$.  All results above apply to this more general setting.  We just need to mention that the constructed intrinsic minimal disc $Z=Z_u$ is a countably $2$-rectifiable set. The only difference from the special case of the Hausdorff area discussed above is that the constant $\frac \pi 4$ appearing
in \eqref{Z-hausdorff-growth} and in (iii) of \tref{thmB},  is replaced by a constant
$q(\mu) \in [\frac 1 2, 1]$  depending on the definition of area $\mu$.
This constant  $q(\mu )$ is maximal  $q(\mu^i )=1$ for the inscribed Riemannian area $\mu ^i$, making the corresponding statements of
\tref{thmA} and \tref{thmB} optimal in this case.

All definitions of area  coincide for all Sobolev maps with values in $X$  if   the space $X$ has the so-called  property (ET) discussed in \cite{LW}.  This is the case for
many geometrically significant spaces such as spaces with one-sided curvature bounds in the sense of Alexandrov,
sub-Riemannian manifolds or infinitesimally Hilbertian spaces with lower Ricci curvature bounds.
Thus if $X$ satisfies the property (ET), one can replace  the factor $\frac \pi 4$ appearing in    \eqref{Z-hausdorff-growth} and in (iii) of \tref{thmB} by
$1=q(\mu^i )$.

Moreover, if $X$ satisfies  property (ET) then any minimal disc $u\in \Lambda (\Gamma ,X)$ is conformal and   the map
$P\in W^{1,2} (D,Z)$ is conformal as well. Without  property (ET) the map $P$ is  $\sqrt 2$-quasiconformal and the constant $\sqrt 2$ is optimal.
We emphasize, that the quasiconformality  and conformality is understood here in the infinitesimal almost everywhere sense (Subsection \ref{subsec:infinite} and \cite{LW}) and does not imply that $P$ is a homeomorphism.

\subsection{Absolute minimal fillings}
Our results apply to the problem of finding  a  disc  realizing the infimum of Gromov's  restricted minimal filling area problem, cf. \cite{Iva08}.  Let $\mu$ be a definition of area. Let $(\Gamma,d_0)$ be a metric space biLipschitz equivalent to the unit circle $S^1$.
The restricted filling $\mu$-area of $\Gamma$ is defined as
$$m_{\mu} (\Gamma ):=\inf \{\mu ( M) \},$$
 where $M$ runs over all smooth Finsler metrics on the disc $\bar D$ such that for the induced distance function on $\partial M$ one has a $1$-Lipschitz homeomorphism
$\partial M\to \Gamma$.
Using the solution of the absolute Plateau problem in \cite{LW} and the results in the present paper we  get:

\begin{thm}  \label{fillingareathm}
Let $(\Gamma, d_0)$ be a biLipschitz circle and  let $\mu$ be a quasi-convex definition of area. Then   the restricted filling area $m_{\mu} (\Gamma )$ is equal to the  Sobolev filling area   defined as
 $$m_{\mu,Sob} (\Gamma)=\inf \{Area _{\mu} (u)\; : \;Y \, \text{complete}, \;
\Gamma \subset  Y \; ,\;  u\in \Lambda (\Gamma,Y)\}.$$
There exists a
compact, geodesic, countably $2$-rectifiable  metric space $Z$ homeomorphic to $\bar D$  such that
$\mu (Z) =m_{\mu} (\Gamma)$ and there exists a map $P:\bar D\to Z$ such that the conclusions of \tref{thmA}, \tref{thmB} and \tref{arc-chord}  hold true with $C=\frac 1 {2\pi}$, $l_0=\infty$,
and with the constant $\frac \pi 4$ replaced by $q(\mu )$.
 Moreover, there exists a $1$-Lipschitz arclength preserving homeomorphism $(\partial Z,d_Z)\to (\Gamma ,d_0)$.
\end{thm}

\subsection{A useful technical result}
We  mention a technical achievement of the paper.
The \emph{geometrically obvious} fact that the restriction of a minimal  disc to an open subdisc is again a minimal disc is indeed non-trivial due to two problems: the boundary of the subdisc might be wild, and even if it is smooth,
the restriction of $u$ to this boundary might be very far from being Lipschitz continuous. Both problems are solved  in  Section \ref{sec:quadr}.
The main implication for the present paper is the following seemingly obvious but technically non-trivial statement.

\begin{prop} \label{newprop}
Let  $X$ admit a $(C,l_0)$-isoperimetric inequality for the definition of area $\mu$. Let
  $u:\bar D\to X$ be a  solution of the Plateau problem as above. Let $T\subset   \bar D$ be a Jordan curve with Jordan domain $O$.
  If the curve $u|_T$  has finite  length $l\leq l_0$  then
$\Area _{\mu} (u|_O) \leq C l^2$.
\end{prop}
If the boundary curve $T$ and the restriction of $u$ to $T$ are sufficiently regular the proof of \pref{newprop} is simple \cite{LW}, Lemma 8.6.

\subsection{Possible variations of the construction} \label{subsec:subtle}
It is possible to define the intrinsic metric structure  for a minimal disc $u:\bar D\to X$  in a slightly different way. Namely, we can  restrict the set of
curves $\gamma$ in the definition of the pseudo-metric $d_u$ to any of the following families of curves: rectifiable, piecewise biLipschitz, or  piecewise smooth curves in $\bar D$.
 Unlike the smooth situation, the arising pseudo-metric
and thus the associated metric space $Z$ may depend on the choice of the family, even if $u$ is Lipschitz continuous,  cf. \cite{Pet-intrinsic}, Section 4 for related discussions.
  However, for any of these choices of the family of curves and the corresponding  associated metric space $Z$,  all the theorems stated above remain valid. The last statement of \tref{prop:first} is then only true for curves $\gamma$ in the corresponding family.  The proofs remain the same, see Section~\ref{subsec:changes-different-curve-families} for some remarks.

It is possible to view the space $Z$ (and  the variants of $Z$ constructed via  different families of curves as above) from another classical perspective. The map $u$ induces an $L^2$-field  $z\mapsto \apmd u_z$ of seminorms on $D$, the analogue of the
pull-back of the Riemannian metric, see Subsection \ref{subsec:apmd}. The length of \emph{almost} any curve $\gamma \subset D$ with respect to the pseudo-distance $d_u$
can be computed by the same
 formula as in  Finsler geometry, using this measurable field of seminorms.  Thus, the space $Z$ is \emph{almost} defined by the  approximate metric differentials of $u$.

\subsection{Structure of the paper}
  Sections \ref{sec:prelim}, \ref{sec:Sob},  \ref{sec:topo} and \ref{sec:quadr}, consist of  preliminaries and preparations. A reader familiar with the subject
may skip these sections.
In Section \ref{sec:prelim} we collect preliminaries from metric geometry, including  definitions of area, as well as area and co-area formulas for rectifiable sets. In Section \ref{sec:Sob} we collect some basics about  Sobolev maps from domains in $\R^2$ to metric spaces. In    Section \ref{sec:topo} we recall several statements from classical two-dimensional topology  related
to the Jordan curve theorem. In Section \ref{sec:quadr} we deal with fillings of badly parametrized curves and gluings of Sobolev maps on non-regular domains,
preparing the proof of \pref{newprop}.

 In Section \ref{sec:Plateau} we recall the existence and regularity results for minimal discs from \cite{LW} and prove  \pref{newprop}.  We slightly reformulate the regularity statements from \cite{LW}, emphasizing that the Hoelder continuity is controlled via an estimate of the lengths of some image curves, thus giving us  control over the intrinsic structure of the minimal disc.
In Section \ref{sec:metric} we fix a minimal disc $u:\bar D\to X$, associate to $u$ a metric space
$Z$ as in Subsection  \ref{subsec:second} above and prove \tref{prop:first}.
  We  observe that the approximate metric differentials of $u$ and $P$ coincide at almost all points of $D$. In particular, this shows that $P$ is as regular as $u$.
This includes all  statements of \tref{thmB} except the first and main one that $P$ is a uniform limit of homeomorphisms. Moreover, it implies that $Z$ is   a countably $2$-rectifiable space.  What remains to be controlled  are the  topological and isoperimetric properties of $Z$.
  In Section \ref{sec:inner-structure-mindiscs} we prove the topological and isoperimetric properties of  the space $Z$ stated in \tref{thmA} and \cref{cor:neu}, using  classical results from $2$-dimensional topology.
 In Section \ref{sec:harvest} we collect everything proven so far and finish the proof of the  main results.
 In   Section \ref{sec:absolute} we discuss the absolute filling problem and prove \tref{fillingareathm}.
In the last section  we collect examples mentioned above and some natural questions about the structure  of minimal discs.

\vspace*{1ex}

{\bf Acknowledgements:} We would like to thank Robert Young for a discussion on \cite{Papasoglu}  which
inspired (iii) of \cref{cor:neu} above.  We thank Heiko von der Mosel, Anton Petrunin and Stephan Stadler for helpful comments and discussions.

\section{Preliminaries}\label{sec:prelim}

\subsection{Basic notation}
The following notation will be used throughout the paper.
The Euclidean norm of a vector $v\in\R^n$ is denoted by $|v|$.  We denote the open unit disc in $\R^2$ by $D$.  A \emph{domain} will always mean an open, bounded, connected subset of $\R^2$.

 Metric spaces appearing in this paper will be assumed complete.
A metric space is called proper if its closed bounded subsets are compact. We will denote distances in a metric space $X$ by $d$ or $d_X$.
Let $X=(X,d)$ be a metric space. The open ball  in $X$ of radius $r$ and center $x_0\in X$ is denoted by
$$B(x_0,r)=B_X(x_0,r) = \{x\in X: d(x_0, x)<r\}. $$
%
A Jordan curve in $X$ is a subset $\Gamma\subset X$ which is homeomorphic to $S^1$. Given a Jordan curve $\Gamma\subset X$, a continuous map $c\colon S^1\to X$ is called a weakly monotone parametrization of $\Gamma$ if $c$ is the uniform limit of homeomorphisms $c_i\colon S^1\to\Gamma$.
For $m\geq 0$, the $m$-dimensional Hausdorff measure on $X$ is denoted by $\hm^m=\hm^m_X$. The normalizing constant is chosen in such a way that on Euclidean $\R^m$ the Hausdorff measure $\hm^m$ equals the Lebesgue measure $\mathcal L^m$.
By $\mathfrak S_2$ we denote the proper metric space of seminorms on $\R^2$ with the distance given by
$d_{\mathfrak S_2} (s,s')= \max _{v\in S^1} \{|s(v)-s'(v)| \}$.

\subsection{Rectifiable  curves}
Let $X=(X,d)$ be a metric space.
The \emph{length} of a (continuous)  curve $c\colon I\to X$, defined on an interval $I\subset\R$, is given by
\begin{equation} \label{eq:rect}
 \length_X(c):= \sup\left\{\sum _{i=1} ^{k+1} d(c(t_i), c(t_{i+1})): t_i\in I, t_1<\dots<t_{k+1}\right\}.
\end{equation}
The definition  extends to continuous curves defined on $S^1$.
 A continuous curve of finite length is called \emph{rectifiable}.

If $c\colon I=[a,b]\to X$ is a rectifiable curve of length $l$ then the length function of
$c$ is the continuous monotone map $s:I=[a,b]\to [0,l]$ given by $s(t)=\length (c|_{[a,t]})$.  The curve $c$ is parametrized by arclength if $s:I\to [0,l]$ is an isometry.
 The curve $c$ has the form $c_0\circ s$, where $c_0:[0,l] \to X$ is the arclength parametrization of $c$.

A \emph{geodesic} is an isometric embedding on an interval.  A space $X$ is called \emph{a geodesic space} if any pair of points in $X$ is connected by a geodesic.  A space $X$ is a \emph{length space}
if for all $x,y\in X$ the distance $d(x,y)$ coincides with $\inf \{\ell_X (c) \}$, where
$c$ runs over the set of all curves connecting $x$ and $y$.  A proper length space is a geodesic space by the theorem of Hopf-Rinow.

 A rectifiable curve $c$
is called \emph{absolutely continuous} if it sends subsets of $\mathcal H^1$-measure $0$ in $\R$
 to subsets of $\mathcal H^1$-measure $0$ in $X$.  Equivalently, the length function
$s$ of $c$ is contained in the Sobolev space $W^{1,1} (I)$.  In this case we have
$\ell(c)= \int _I s'(t) dt$. Moreover,  for almost all
$t\in I$, the value $s'(t)$ is the \emph{metric differential} of $c$ at $t$, thus
\begin{equation} \label{eq:length}
s'(t) =\lim _{\epsilon \to 0} \frac {d(c(t), c(t+\epsilon))}{|\epsilon |}.
\end{equation}
For a   Borel function  $f:c\to [0,\infty]$ we set as usual
$$\int _c f:=  \int _0 ^l f(c_0 (t)) dt.$$

A rectifiable curve $c:I\to X$ with length parametrization $s:I\to \R$ is in the Sobolev space $W^{1,p} (I,X), 1\leq p <\infty$, if and only if $s$ is in the classical space $ W^{1,p} (I,\R)$. A concatenation
of Sobolev curves is a Sobolev curve.  If $T$ is a metric space homeomorphic to an interval or a circle
and $u:T\to X$ is a continuous map, we can unambiguously talk about the length  of the curve $u(T)$,
since it does not depend on the special  parametrization of $T$ by an interval.
If  $T$ is a metric space  biLipschitz equivalent to
an interval or a circle and $u:T\to X$ a continuous map to a metric space $X$ we say that
$u$ is in the Sobolev class $W^{1,2} (T,X)$ if for the arclength parametrization  $c :I\to T$ of $T$ we have $u\circ c \in W^{1,2} (I,X)$.

A continuous curve $c\colon I\to X$ is called a piecewise biLipschitz curve if there exists a partition of $I$ into a finite number of subintervals such that the restriction of $c$ to each subinterval is a biLipschitz map.

A Jordan curve $\Gamma$ is a \emph{chord-arc} curve if $\Gamma$ is biLipschitz homeomorphic to $S^1$.
Due to \cite{Tuk80} a (connected, bounded) domain $\Omega \subset \R^2$ is a \emph{Lipschitz domain} if and only if
$\Omega$ is a bounded component of $R^2 \setminus T$, where $T \subset R^2$ is a finite union of pairwise disjoint chord-arc curves.


\subsection{Definitions of  area}\label{sec:def-vol-normed}
While there is an essentially unique natural way to measure areas of Riemannian surfaces, there are many  different ways to measure areas of Finsler surfaces,
some of them more appropriate for different questions.  We refer the reader to
\cite{Iva08}, \cite{Ber14}, \cite{LW-energy-area}, \cite{AlvT04} and the literature therein for  more information.

A definition of area $\mu$ assigns a multiple $\mu _V$ of $\mathcal H^2$ on any $2$-dimensional normed space $V$, such that natural monotonicity assumptions are fulfilled, cf. \cite{AlvT04}.  In particular, it assigns the number
$\J^{\mu} (s)$, \emph{the $\mu$-Jacobian}, to any seminorm $s$ on $\R^2$ in the following way. By definition,  $\J ^{\mu} (s)=0$ if the seminorm is not a norm. If $s$ is a norm then  $\J ^{\mu} (s)$ equals the $\mu _{(\R^2,s)}$-area  of the unit   Euclidean square  in  $ \R^2$. Indeed, a choice of a definition of area $\mu$ is equivalent to a choice of a \emph{Jacobian} $\J ^{\mu} :\mathfrak S_2 \to [0,\infty)$ which satisfies natural transformation and monotonicity conditions, cf. \cite{LW-energy-area}, Section 2.3.

Any two definitions of area  differ at most by a factor of $ 2$.   The largest definition of area is the inscribed Riemannian definition of area $\mu ^i$, introduced in  \cite{Iva08}. Other  prominent examples are the Busemann definition
$ \mathcal H^2$,  the Holmes-Thompson definition $\mu ^{ht}$,   Gromov's $mass^*$-definition $m^{\ast}$. We refer to \cite{AlvT04}, \cite{Iva08}
 for  a thorough discussion of these examples and of the whole subject and to \cite{LW-energy-area} for a detailed description of the
 corresponding Jacobians.

For a definition of area $\mu$,
the number $q(\mu)\in [\frac 1 2, 1]$ appearing in the main theorems of this paper, cf. Subsection \ref{subsec:genarea},
   is defined to be the maximal number $q$ such that $\mu _V\geq q \cdot \mu^i _V$ holds true on any normed plane $V$.
 Thus  $q(\mu)$ equals to
the infimum of the quotient $\J^{\mu}  (s)  /\J^{\mu^i} (s)$ taken over all norms $s$. In particular, cf. \cite{LW-energy-area},
 $$q(\mu^i)=1,  \quad q(\mathcal H^2)= \frac \pi 4,   \quad q(\mu^{ht} ) =\frac 2 \pi.$$


\subsection{Lipschitz maps and rectifiable sets} \label{subsec:rect}
Given a measurable  subset $K\subset \R^2$
and a Lipschitz map $f:K\to X$ to a metric space $X$, we say that a seminorm $s:\R^2
\to [0,\infty)$
is the metric differential of $f$ at the point $z\in K$ and denote it by $\md _z f$
if
\begin{equation} \label{def-md}
\lim _{y\to z} \frac {d(f(z),f(y)) -s(y-z)} {|y-z|} =0.
\end{equation}
Any Lipschitz map defined  on a measurable subset  $K\subset \R^2$ has a uniquely defined metric differential at almost every point \cite{Kir94}.   Moreover, for any $\epsilon >0$,  the set $K$ can be decomposed as a disjoint union $K=S\cup _{1\leq i<\infty} K_i$
such that the following holds true. The set $S$ is the union of a set of $\mathcal H^2$-area $0$ and the set of all points at which the metric differential is not a norm. The sets $K_i$ are compact.
For any  $i$,
the restriction $f:K_i \to f(K_i) \subset X$ is $(1+\epsilon)$-biLipschitz, if $K_i$ is endowed with the distance induced by the norm $s_i =\md_z f$, for an arbitrary  $z\in K_i$.  Finally, $\mathcal H^2(f(S))$
is zero, see \cite{Kir94}.

Recall that a metric space $X$ is called \emph{countably $2$-rectifiable} if up to a subset $S\subset X$
of $\mathcal H^2$-measure $0$,  $X$ is a countable union of
Lipschitz images $f_i(K_i)$  of compact subsets of $K_i\subset \R^2$.
The above decomposition result shows that up to a set of $\mathcal H^2$-area $0$, any rectifiable set is a disjoint union of pieces which are arbitrary biLipschitz close to compact subsets
of normed planes \cite{Kir94}.
From this it  follows that
any definition of area $\mu$ provides a measure $\mu _X$   on any   countably $2$-rectifiable set $X$ uniquely determined by the following properties. The measure  $\mu _X$ is absolutely continuous with respect to $\mathcal H^2 _X$, on  Borel subsets of normed planes $\mu _X$ is defined  as  above and, finally,
 any  $1$-Lipschitz map between rectifiable sets does not increase the $\mu$-area, cf. \cite{Iva08}.

  The decomposition result above yields a way  how to compute the $\mu$-area of the image of a Lipschitz map and thus of any rectifiable set,
 cf.  \cite[Theorem 7]{Kir94}, \cite{Iva08}:
\begin{lem}  \label{lem:area}
Let $K\subset \R^2$  be  measurable, let $f:K\to X$ be a Lipschitz map with
$Y=f(K)$  and let $\mu$ be a definition of area.   Let $N:Y\to [1,\infty]$ be the multiplicity  function $N(y)= \#\{z\in K: f(z) = y\}$.  Then:
 \begin{equation*}  \label{eq:area+}
  \int_Y N(y) d\mu _Y(y) = \int_{K} \jac ^{\mu}(\md _z f)\, dz.
 \end{equation*}
\end{lem}

\subsection{Co-area inequality with respect to $\mu$}
If $X$ is a measurable subset of   $\R^2$ and $f:X\to \R$ a $1$-Lipschitz function then the classical co-area formula
(\cite{Fed69}, Theorem 3.2.22})  implies
\begin{equation} \label{eq:Riemannian}
\mathcal H^2 (X) \geq    \int _{\R}  \mathcal H^1 (f^{-1} (t)) \, dt.
\end{equation}

In the realm of metric spaces one needs to insert an (optimal) factor of $\frac  \pi 4$ on the right side,  cf. \cite{Fed69}, Theorem 2.10.25, 2.10,26:
\begin{lem} \label{lem:fed}
For any proper metric space $Y$, any Borel subset $X\subset Y$ with finite $\mathcal H^2 (X)$ and any $1$-Lipschitz function
  $f:X\to \R$ we have
  $$ \mathcal H^2 (X) \geq   \frac {\pi} 4 \cdot  \int _{\R}  \mathcal H^1 (f^{-1} (t)) \, dt \; .$$
\end{lem}

Note that the factor $\frac \pi 4$ coincides with the number  $q(\mathcal H^2)$ introduced in  Subsection \ref{sec:def-vol-normed}. Thus under the assumption that $X$ is countably $2$-rectifiable   the lemma above  is a special case of the following:

\begin{lem} \label{co-area-gen}
Let  $X$ be a countably $2$-rectifiable set and let $f:X\to \R$ be a $1$-Lipschitz function.
Then
 $$\mu (X) \geq q(\mu)  \cdot  \int _{\R}   \mathcal H^1 (f^{-1} (t)) \, dt.$$
\end{lem}

\begin{proof}
By definition of $q(\mu)$, we have $q(\mu) ^{-1} \cdot \mu \geq \mu ^i$, where $\mu ^i$ is the inscribed Riemannian area.  Hence we only need to show
\begin{equation} \label{mui}
\mu^i (X) \geq \int _{\R}   \mathcal H^1 (f^{-1} (t)) \, dt.
\end{equation}

Rademacher's theorem \cite{Kir94} and the decomposition of $X$ into small pieces  approximated by pieces of normed spaces
as in Subsection  \ref{subsec:rect}  show that it suffices to prove the result in the case  that
$X$ is a subset of a normed plane $(\R^2,s)$ and $f:\R^2\to \R$ is  a linear map, cf. \cite{AKrect00}, Section 9.
Let $s_0$ be the Euclidean norm whose unit ball is the Loewner ellipse of the unit ball of $s$.  Then $s_0 \geq s$, hence
$f:(\R^2,s_0)\to \R$ is still $1$-Lipschitz. Moreover,  by the definition of the inscribed Riemannian area, the $\mu^i$-area
on $(\R^2, s)$ coincides with the Lebesgue area of the Euclidean plane $(\R^2,s_0)$.  Thus \eqref{mui} follows from   \eqref{eq:Riemannian}.
\end{proof}

\begin{rmrk}
The co-area factor in \lref{co-area-gen}  is optimal for some definitions of area $\mu$, but not for all.
It can be shown, that the optimal $\mu$-dependent co-area constant in \lref{co-area-gen} equals $\frac 4 v$,
where $v$ is the $\mu$-area of the unit ball in the  plane $(\R^2,s_{\infty})$ with the sup-norm.
\lref{co-area-gen} is optimal for the inscribed Riemannian, Hausdorff and Holmes-Thompson, but not for the
Benson definition of area.
\end{rmrk}

\section{Sobolev maps} \label{sec:Sob}
\subsection{Exceptional families of curves} \label{subsec:except}
We refer to \cite{HKST15} for more details.
Let $\Omega \subset \R^2$ be a bounded  domain  and   let $1<p<\infty$ be given.
A family $\mathcal E$ of curves in
$\Omega$ is called \emph{$p$-exceptional}  if there exists a
Borel function $f:\Omega \to [0,\infty]$ such that $f\in L^p (\Omega )$ and  $\int _{\gamma} f =\infty$ for all rectifiable curves
$\gamma $ in the family $\mathcal E$.  Note that the set of all non-rectifiable curves is exceptional by definition.  We say that a property holds true for $p$-\emph{almost every curve} if the set of curves for which the property fails  is $p$-exceptional.

A countable union of $p$-exceptional families is $p$-exceptional.
If $S$ is a subset of $\Omega $ with $\mu (S)=0$ then for  $p$-almost all curves $\gamma$ we have   $\mathcal H^1 (\gamma \cap S) =0$.
By Fubini's theorem, a set of non-constant curves in
$\Omega $ parallel to a given line $l$ is
$p$-exceptional for some and then for any $p \in (1,\infty) $ if and only if the projection of the set of curves to the orthogonal line $l^{\perp}$ has $\mathcal H^1$-measure $0$.

\subsection{Generalities on Sobolev maps} \label{subsec:genera}
We assume some experience with Sobolev maps  with values in complete metric spaces $X$  and refer  to \cite{HKST15},
\cite{LW}
and the literature therein.  In this paper we consider only Sobolev maps   defined on   open bounded domains $\Omega \subset \R^2$, intervals  and circles.  In \cite{LW} we  worked with the Soboev spaces $W^{1,p} (\Omega,X)$ as defined in \cite{KS93}. For the present work it is more natural to stick to  the Newton-Sobolev spaces as defined in \cite{HKST15}. Recall that  both notions are equivalent.
More precisely, every map in the Newton-Sobolev space
$N^{1,p} (\Omega ,X)$ is contained in $W^{1,p} (\Omega ,X)$ and any element
$u\in W^{1,p} (\Omega ,X)$ has a representative in $N^{1,p} (\Omega, X)$, uniquely defined up to some  $p$-exceptional subset. Most maps appearing in this paper are continuous, and a continuous map $u\in W^{1,p} (\Omega ,X)$ is automatically in $N^{1,p} (\Omega,X)$.  Thus the difference is not visible in the cases important in this paper.  Therefore, we will freely interchange between
$N^{1,p}$ and $W^{1,p}$.

The space $L^p(\Omega, X)$ consists of those measurable and essentially separably valued maps $u : \Omega\to X$ for which the composition $f\circ u$ with the distance function $f$ to some point in $X$ is in the classical space $L^p(\Omega)$. A map $u:\Omega \to X$ is in the Newton-Sobolev space $N^{1,p} (\Omega,X)$
if $u\in L^p (\Omega, X)$ and  if there exists a Borel  function $\rho \in L^p (\Omega )$,
such that for $p$-almost all curves $\gamma: I \to \Omega$  the composition $u\circ \gamma$ is a continuous curve and
  the following  inequality holds true:
\begin{equation} \label{eq-n1p}
\int _{\gamma} \rho  \geq \ell_X(u\circ \gamma ).
\end{equation}

 Up to sets of measure $0$, there exists a uniquely defined minimal function $\rho _u$ satisfying the condition above. It is called  the \emph{generalized gradient} of $u$.  The integral $\int _{\Omega} \rho ^p _u (z) dz$ coincides with
   the  Reshetnyak energy $E_+ ^p (u)$ of $u$, see \cite{HKST15}, Theorem 7.1.20 and \cite{LW}, Section 4.

\subsection{Approximate metric differentials} \label{subsec:apmd}
Let $u \in W^{1,p} (\Omega ,X)$ be as above.  Then $u$ has an \emph{approximate metric differential}
at almost every point $z\in \Omega$. This approximate metric differential is a
seminorm $s$ on $\R^2$,  denoted by  $\apmd u_z$, which satisfies \eqref{def-md}, where  $\lim$ is replaced by the approximate limit $\ap \lim$.
   We refer the reader to \cite{Kar07},\cite{LW}, Section 4
 and recall here only the following two structural results that a posteriori  could be taken as the definition of the approximate metric differential.
The field  of seminorms $z\mapsto \apmd u_z$ is a measurable map contained in $L^p(\Omega, \mathfrak S_2)$, thus, changing this map  on a subset of measure $0$, we may assume that
 $z\mapsto \apmd u _z$ is an everywhere  defined Borel map.
There is a countable, disjoint decomposition $\Omega =S\cup _{1\leq i<\infty} K_i$ into a set $S$ of zero measure and compact subsets $K_i$ such that the following holds true. The restriction of $u$ to any $K_i$ is Lipschitz continuous,  the  metric differential  of the restriction $u:K_i\to X$ exists at any  $z\in K_i$ and  coincides with $\apmd u_z$.

\subsection{Length of almost all curves}
Let $u\in N^{1,p} (\Omega, X)$ be given.  Then for $p$-almost all rectifiable  curves $\gamma  :I\to  \Omega$ parametrized by arclength, the composition
$u\circ \gamma$ is absolutely continuous (\cite{HKST15}, Prop. 6.3.2). Using  approximate metric differentials
we can compute the length of almost all  curves by the usual formula:

\begin{lem} \label{almostall}
Let $u\in N^{1,p}  (\Omega ,X)$ be given.
   Then for $p$-almost all   rectifiable curves   $\gamma :I\to \Omega$ parametrized by arclength we have
   \begin{equation} \label{eq:almostall}
\ell_X(u\circ \gamma )=\int _I \apmd u_{\gamma (t)}  (\dot {\gamma }(t)) dt.
\end{equation}
\end{lem}

\begin{proof}
Fix a Borel representative of the map $z\mapsto \apmd u_z$.  For any $\gamma:I\to \Omega$ parametrized by arclength
the integrand on the right hand side of \eqref{eq:almostall} is measurable, hence the right hand side is well-defined.
Consider the decomposition  $\Omega =S\cup K_i$ described above, such that $S$ has measure  $0$. Moreover, the sets $K_i$ are compact,    the restrictions  $u:K_i\to X$ are Lipschitz continuous and have metric differentials at all points. Finally, these metric differentials
 coincide with $\apmd u_z$ at \emph{all} $z\in K_i$.
 The set of curves
$\gamma:I\to \Omega$ whose intersection with $S$ has non-zero $\mathcal H^1$-measure
is $p$-exceptional.  The lemma follows, once we have shown \eqref{eq:almostall}
  for all $\gamma :I\to \Omega$ outside this $p$-exceptional set and such that
$u\circ \gamma$ is absolutely continuous.

 For any $\gamma$ as above set $I_i=\gamma ^{-1} (K_i) \subset I$ and let
$J_i$ be the set of all Lebesgue points of $I_i$ in $I$, at which $\gamma$ has a differential and the absolutely  continuous curve $u\circ \gamma$ has a metric differential.
By assumption on $\gamma$,  the union $J=\cup J_i$ has full measure in $I$. On the other hand, for any $t\in J$ the metric differential
$\md _t (u\circ \gamma ) (1)$ must
coincide with $\apmd u _{\gamma(t)} (\dot \gamma (t))$.   Thus, integrating this equality and using \eqref{eq:length} we obtain \eqref{eq:almostall}.
\end{proof}

We deduce as consequences:

\begin{cor} \label{twomaps}
Let $X^{\pm}$ be complete metric spaces. Let $u^{\pm} \in L^p(\Omega,  X^{\pm})$ be two  maps.
  Assume that for
$p$-almost all curves $\gamma \subset  \Omega $ the compositions $u^{\pm}\circ\gamma$ are continuous and
\begin{equation} \label{eq:twomaps}
\ell_{X^+}(u^+\circ \gamma )= \ell_{X^-}(u^-\circ \gamma ).
\end{equation}
    Then $u^+\in N^{1,p} (\Omega, X^+)$
if and only if $u^- \in N^{1,p} (\Omega ,X^-)$. In this case, the approximate metric differentials $\apmd u^{\pm}$ of $u^{\pm}$
coincide almost everywhere.
\end{cor}

\begin{proof}
The first claim follows directly from the defintion of Newton-Sobolev spaces $N^{1,p}$.

Fix a unit vector $v\in\R^2$. Then for almost every $z\in\Omega$ we have \eqref{eq:twomaps}
 and  \eqref{eq:almostall} for all sufficiently short segments $\gamma$ centered at
$z$ and in direction $v$. The Lebesgue Differentiation theorem thus
implies that $\apmd u^+_z(v) = \apmd u^-_z(v)$ for almost every
$z\in\Omega$. Applying this to a countable dense set of directions $v$
implies that the measurable fields of seminorms  $\apmd u^\pm$ are almost everywhere equal.
\end{proof}


\subsection{Traces and gluings}
Let $\Omega \subset \R^2$ be a Lipschitz domain with boundary $\partial \Omega$ and let $u\in W^{1,p} (\Omega,X)$ be given.
Then $u$ has a well-defined trace $\trace u \in L^p(\partial \Omega ,X)$ with the following property, cf. \cite{KS93}.
  For the distance function  $f:X\to \R$ to any point $x\in X$, we have
	$\trace (f\circ u) =f\circ \trace u  \in L^p (\partial \Omega)$, where
   on the left hand side the usual trace of Sobolev real-valued functions is considered.
	
	Let a curve $S\subset \Omega $ separate the Lipschitz domain  $\Omega $
 into two Lipschitz subdomains $\Omega ^{\pm}$. If $u^{\pm} \in W^{1,p} (\Omega ^{\pm},X)$ have the same trace on $S$ then $u^{\pm}$  define together
 a map $u\in W^{1,p} (\Omega,X)$, see \cite{KS93}.

\subsection{Area of Sobolev maps} \label{subsec:areasob}
Let $\mu$ be  a definition of area  and consider the corresponding Jacobian
$\J^{\mu} :\mathfrak S_2\to [0,\infty ) $, see Subsection \ref{sec:def-vol-normed}.  For  $u\in W^{1,2} (\Omega, X)$
the $\mu$-area of $u$ is defined  by
\begin{equation} \label{eq:area-sob}
\Area_{\mu} (u)  :=\int _{\Omega}   \J^{\mu} ( \apmd  u_z) \, dz.
\end{equation}

The number $\Area _{\mu} (u)$ is finite  and satisfies $\Area _{\mu} (u) \leq E_+^2 (u)$, \cite{LW}, Lemma 7.2.
In view of the area formula for Lipschitz maps this is a natural extension
of the parametrized $\mu$-area to Sobolev maps.

Recall that for any $p>2$ any map $u\in N^{1,p}  _{loc}(\Omega ,X)$ is continuous and has \emph{Lusin's property (N)}, thus it sends
$\mathcal H^2$-zero sets to $\mathcal H^2$-zero sets. From the decomposition of $\Omega$ into parts on which $u$ is Lipschitz and a set
of measure $0$, we deduce that
the image $u(\Omega)$ is countably $2$-rectifiable  and of finite Hausdorff area. More precisely we deduce  from \lref{lem:area} (cf. \cite{Kar07}):

\begin{lem} \label{area-sob}
Let $u\in N^{1,2} (\Omega, X)$ be  a continuous map with Lusin's property (N). Let $K$ be a measurable subset of $\Omega$ and $Y=u(K)$.  Let
 $N:Y\to [0,\infty]$ be the multiplicity function $N(y)= \#\{z\in \Omega: u(z) = y\}$.  Then the following holds true:
\end{lem}

 \begin{equation*}  \label{eq:area++}
  \int_Y N(y) \, d\mu _Y(y) = \int_{K} \jac ^{\mu}(\apmd u_z)\,dz.
 \end{equation*}


\subsection{Special infinitesimal structure}  \label{subsec:infinite}
A seminorm $s \in \mathfrak S_2$ is  \emph{$Q$-quasiconformal} for some $Q\geq 1$ if for all $v,w\in S^1$ the inequality $s(v) \leq Q\cdot s(w)$ holds true. A quasiconformal seminorm is either a norm or the $0$-seminorm.
 A map $u\in W^{1,2} (\Omega, X)$ is called $Q$-quasiconformal if  the seminorms $\apmd u_z$ are $Q$-quasiconformal for almost every $z\in \Omega$.
For $Q=1$ we call such maps conformal.

For every  definition of area $\mu$ and every    $Q$-quasiconformal map $u\in W^{1,2} (\Omega , X)$  we have:
\begin{equation} \label{eq:Q2}
 \Area _{\mu} (u) \geq Q^{-2} \cdot E_+^2 (u).
\end{equation}

The inequality above is the basic ingredient for most regularity results in \cite{LW}. The $\sqrt 2$-quasiconformality of a map
 $u\in W^{1,2} (D,X)$ can be guaranteed  by the following lemma (\cite{LW-energy-area}, Lemma 4.1, cf. \cite{LW}, Theorem 1.2). This  lemma  also  strengthens  \eqref{eq:Q2}
replacing the constant $Q^{-2}=\frac 1 2 $
by the $\mu$-depending constant $q (\mu ) \in [\frac 1 2, 1]$.

\begin{lem} \label{lem:Qq}
Let $u\in W^{1,2} (D,X)$ be such that $E_+^2 (u) \leq E_+^2 (u\circ \phi )$ for all biLipschitz homeomorphisms $\phi:\bar D\to \bar D$.
Then $u$ is $\sqrt 2$-quasiconformal. Moreover, for any definition of area $\mu$ and any subdomain $\Omega \subset D$ we have
$$\Area _{\mu} (u|_{\Omega}) \geq   q(\mu) \cdot  E_+^2 (u|_{\Omega}) .$$
\end{lem}

A space $X$ satisfies property (ET) if for any map $u\in W^{1,2} (D,X)$ the  seminorm $\apmd u_z$ is either degenerate or comes from some Euclidean product at almost every point $z\in D$.  We refer to \cite{LW} for a discussion of property (ET) and the classes of spaces satisfying it, see also Subsection
\ref{subsec:genarea}. We just recall here that if $X$ satisfies  property (ET) then for any $u\in W^{1,2} (D,X)$ the $\mu$-area of $u$ does not depend
on the definition of area $\mu$.  Moreover, in this case any map $u$ satisfying the assumption of \lref{lem:Qq} is conformal, and therefore the equality
$\Area _{\mu} (u|_{\Omega}) = E_+^2 (u|_{\Omega})$ holds for any subdomain $\Omega $ of $D$.

\subsection{Lengths of circles} We will need a variation of the classical Lemma of Courant-Lebesgue, cf. \cite{LW}, Lemma 7.3.
Let $X$ be a complete metric space and let $u\in N^{1,2} (D,X)$. Let $x\in \bar D$ be an arbitrary point. For $t\leq 1$ denote  by
$S_t=S_t(x)$ the set  $S_t (x) =\{z\in D : |z-x|=t \}$, which is either a circle or a circular arc.  For almost all $t\leq 1$ the restriction $u:S_t\to X$ is
absolutely continuous.  Denote by $l_t$ the length of $u:S_t \to X$.

\begin{lem} \label{lem:ballcircle}
In the above notations
let $1>r>0$ be given and set $e_r =E_+^2 (u|_{B(x,r)\cap D})$.  Then there exists a set $S$   of positive measure in  the interval $[\frac 2 3 r, r]$, such that for any $t\in S$ we have  $l_t^2 \leq 6\pi \cdot e_r$.
\end{lem}

\begin{proof}
For the generalized gradient  $\rho _u$ of $u$, we integrate $\rho _u ^2$ in polar coordinates around $x$, use  Hoelder's inequality and \eqref{eq-n1p}:
\begin{equation} \label{eq:cour}
e_r = \int _0 ^r \left(\int _{S_t} \rho _u^2 \right)\, dt \geq \int _0 ^r   \frac 1 {2\pi t} \cdot \left(\int _{S_t} \rho _u \right)^2 \, dt \geq \frac 1 {2\pi r} \int _0 ^r l_t ^2 \,dt.
\end{equation}
If $l_t^2 > 6\pi \cdot e_r$ for almost all $t\in [\frac 2 3 r, r]$ then $e_r >  \frac 1 {2\pi r }\cdot \frac r 3 \cdot 6\pi \cdot  e_r$ which is absurd. This proves the claim.
\end{proof}



\section{Quadratic isoperimetric inequality} \label{sec:quadr}
\subsection{Equal traces in irregular domains}
For a bounded domain $\Omega \subset \R^2$, we denote by $W^{1,2} _0 (\Omega)$ the closure of the set of all smooth functions with compact support in $\Omega$ with respect to the usual norm in  the  Sobolev space $W^{1,2} (\Omega ,\R)$. By continuity, for any $u \in W^{1,2}_0 (\Omega )$ the extension of $u$ by $0$ outside of $\Omega$ defines a function in $W^{1,2} (\R^2,\R)$.    We can now define
\emph{equality of traces} even for irregular domains, when traces are not defined.
\begin{defn}  \label{def:eqtraces}
Let $\Omega\subset \R^2$ be a bounded domain and let $X$ be a complete metric space.
 We say that  $f_1,f_2 \in W^{1,2} (\Omega, \R)$ have equal traces if  $f_1-f_2 \in W^{1,2}_0 (\Omega)$. We say that  $u_1,u_2\in W^{1,2} (\Omega, X)$ have equal traces if for every
 $x\in X$ the compositions of $u_1$ and $u_2$ with the distance function to $x$ have equal traces.
\end{defn}

If $\Omega $ is a Lipschitz domain then  a function $u\in W^{1,2} (\Omega, \R)$ is contained in $W^{1,2} _0 (\Omega)$ if and only if $\trace u :\partial \Omega \to \R$ equals $0$ almost everywhere on $\partial \Omega$. Therefore, for maps $u_{1,2} \in W^{1,2} (\Omega, X)$ on a Lipschitz domain $\Omega$ the maps $u_{1,2}$ have equal traces in the sense of the definition above if and only
if $\trace u_1 =\trace u_2$ almost everywhere on $\partial \Omega$.

We have the following extension of the classical gluing statement:
\begin{lem} \label{lem:bestgluing}
 Let $V\subset \Omega$ be bounded domains in $\R^2$. Let $u\in W^{1,2}  (\Omega, X)$ and  $v\in W^{1,2} (V,X)$ be maps into a complete metric space $X$.
Assume that  $v$ and $u|_V$ have equal traces. Then the map $\hat u$ which equals $u$ on $\Omega \setminus V$ and $v$ on $V$ is contained in $W^{1,2} (\Omega, X)$. Moreover, $\hat u$ and $u$ have equal traces.
Finally, for any definition of area $\mu$, we have  $\Area  _{\mu} (\hat u ) - \Area _{\mu} (u) =\Area _{\mu} (v)- \Area _{\mu} (u|_{V})$.
\end{lem}

\begin{proof}
If $\hat u\in W^{1,2} (\Omega ,X)$ then the approximate metric differentials of $\hat u$ and $u$ must coincide at almost all points of $\Omega \setminus V$. Thus the last claim about the difference
of areas follows  from the  definition \eqref{eq:area-sob} of the  area of Sobolev maps.
In  order to establish the  first two claims it suffices to prove the corresponding statements for compositions of $u$ and $v$ with distance functions to points $x \in X$. Therefore, we may assume that $X=\R$. In this case the difference of $u$ and $\hat u$ equals $0$ outside of $V$ and equals $u-v$ on $V$. Thus  $u-\hat u \in W^{1,2} (\Omega, \R)$, hence $\hat u \in W^{1,2} (\Omega, \R)$.  Moreover, $u-\hat u$ is a
limit in $W^{1,2} (\Omega, \R)$ of a sequence of smooth functions with  support contained in $V$, hence $\hat u$ and $u$ have equal traces.
\end{proof}

\subsection{Conformal changes}
Since the Reshetnyak energy is conformally invariant, it is possible to control the pull-backs of Sobolev maps by conformal diffeomorphisms even if they are not biLipschitz.
\begin{lem} \label{lem:l2enough}
Let $F:\Omega _1\to \Omega _2$ be  a conformal diffeomorphism between bounded domains in $\R^2$. Let $X$ be a complete metric space and $u\in N^{1,2} (\Omega _2,X)$.
Then $v=u\circ F:\Omega _1 \to X$ is measurable and essentially separable valued. The map $v$ is contained in
 $N^{1,2} (\Omega _1 ,X)$ if and only if $v\in L^2 (\Omega _1,X)$.  In that case we have $E_+^2(u)=E_+^2(v)$
 and $\Area_{\mu} (u)=\Area _{\mu} (v)$ for any definition of area $\mu$.
\end{lem}

\begin{proof}
Since $F$ is locally biLipschitz, the
 composition with $F$ defines a bijection $N^{1,2} _{loc} (\Omega _2 ,X)\to N^{1,2} _{loc} (\Omega _1,X)$. In particular, $v:\Omega _1\to X$ is measurable and essentially separable valued. The map $F:\Omega_1\to \Omega _2$ preserves $2$-exceptional families of curves.
Let $\rho \in L^2 (\Omega _2)$ be the minimal generalized gradient of $u$.
Consider $\hat \rho \in L^2 (\Omega _1)$ defined by $\hat \rho ^2(z):=\rho ^2(F(z)) \cdot |det DF(z)|$.
Then for any $O\subset \bar O\subset \Omega _1$, the function  $\hat \rho$ is the minimal generalized gradient of $v|_O \in W^{1,2} (O,X)$,
cf. \cite{LW}, Lemma 6.4.
Therefore, $\hat \rho \in L^2 (\Omega)$ is the minimal generalized gradient of $u\circ F$ in the sense of
\eqref{eq-n1p}, since this is true for all subdomains $O\subset \bar O\subset \Omega _1$.
Thus $v \in N^{1,2} (\Omega _1 ,X)$ if and only if $v\in L^2 (\Omega _1,X)$.  The last equality statement for energies and areas follows from the corresponding statements for the restrictions to
subdomains $O\subset \bar O\subset \Omega _1$, \cite{LW}, Lemma 6.4.
\end{proof}

The condition $u\circ F\in L^2 (\Omega _1,X)$ is automatically fulfilled if the (essential) image of $u$ is contained in a bounded set,
in particular, if $u$ has a continuous extension to $\bar \Omega _2$. In general, $u\circ F$ need not be contained in $L^2 (\Omega_1, X)$.
However, this condition turns out to depend only on the traces:

\begin{lem} \label{lem:pullback}
Let  $F:\Omega _1\to \Omega _2$ be  a conformal diffeomorphism between bounded domains in $\R^2$. Let $u^{\pm}$ be two maps
in $N^{1,2} (\Omega _2,X)$ which have equal traces. If $u^+ \circ F \in N^{1,2} (\Omega_1,X)$ then  $ u^- \circ F \in N^{1,2} (\Omega_1,X)$
and the compositions $u^{\pm} \circ F$ have equal traces.
\end{lem}

\begin{proof}
Taking the composition with  distance functions to  points $x\in X$, we may assume that $X=\R$. Hence, considering the difference $u^+-u^-$, it suffices to prove that $v\circ F \in W^{1,2}_0 (\Omega _1)$ for any $v\in W^{1,2}  _0 (\Omega _2)$.  For any smooth function $v$ with compact support
in $\Omega _2$ the composition $v\circ F$ is smooth, has compact support in $\Omega _1$ and the same energy as $v$.  Since $v$ has compact support in $\Omega _1$ the $L^2$-norm
of $v\circ F$ is bounded by $K\cdot E_+^2 (v\circ F)$ for some constant $K=K(\Omega _1)$, due to the Sobolev inequality. This shows that the composition with $F$ defines a continuous map
$W_0^{1,2} (\Omega _2)\to W_0^{1,2} (\Omega _1)$ and finishes the proof.
\end{proof}

We infer that restrictions of area minimizers to subdiscs minimize the area as well:
\begin{cor} \label{cor:minmin}
Let the image of $u \in N^{1,2} (D, X)$ be  contained in a  bounded set.  Let $F:D\to O$ be a conformal diffeomorphism onto a subdomain $O\subset D$. Then $v:=u\circ F \in N^{1,2} (D,X)$. If $u$
has minimal $\mu$-area among all maps in $N^{1,2} (D ,X)$ with the same trace as $u$  then
 $v$ has minimal $\mu$-area among all maps in $N^{1,2} (D,X)$ with the same trace as $v$.
\end{cor}

\begin{proof}
From \lref{lem:l2enough} and the subsequent remark we deduce $v\in N^{1,2} (D,X)$.  Assume that $v_+ \in N^{1,2} (D,X)$
has the same trace and smaller $\mu$-area than $v$. Then $w:=v\circ F^{-1}$ is contained in $N^{1,2} (O,X)$ and $u|_O$ and $w$ have equal traces
 by \lref{lem:pullback}. Moreover, $w$ has smaller $\mu$-area than the corresponding restriction of $u$.
Now, define $\hat u$ to be equal $u$ on $D \setminus O$ and equal $w$ on $O$. The corresponding map $\hat u$ is in $W^{1,2} (D, X)$, has
 the same trace as $u$ and smaller $\mu$-area, due to \lref{lem:bestgluing}.   This contradicts the minimality assumption.
\end{proof}

\subsection{Quadratic isoperimetric inequality}
Let $\mu$ be a fixed definition of area.  A space $X$ admits a $(C,l_0)$-isoperimetric inequality with respect to $\mu$, if for any \emph{Lipschitz} curve
$\gamma :S^1\to X$ of length $l \leq l_0$ there exists some $u\in W^{1,2} (D,X)$ with trace $\gamma$ and $\Area_{\mu} (u)\leq Cl^2$.
Then   $X$ admits a $(2C,l_0)$-isoperimetric inequality with respect to any other  definition of area $\mu'$.
However, $\mu$-minimal and $\mu'$-minimal discs may be completely  different (\cite{LW}, Prop. 11.6).

In order to avoid many reparametrizations we  state  the following:
\begin{lem} \label{no-rep}
Let $X$ admit a  $(C,l_0)$-isoperimetric inequality with respect to $\mu$.
Let $T$ be a biLipschitz Jordan curve in $\R^2$ and let $\Omega$ be the Jordan domain bounded by $T$.  If $\gamma \in W^{1,2} (T,X)$ is  a curve of length $l\leq l_0$ then
there exists a map $u\in W^{1,2} (\Omega ,X)$ with $\trace u =\gamma$ and such that
$\Area_{\mu} (u) \leq C\cdot l^2$.
\end{lem}

\begin{proof}
The domain $\Omega$ is biLipschitz homeomorpic to $\bar D$ (\cite{Tuk80}). Since compositions with biLipschitz homeomorphisms preserve the class of Sobolev maps, lengths and areas, the statement follows from the corresponding statement in the case $T=S^1$. In the case $T=S^1$ the statement was proved  in \cite{LW}, Lemma 8.5.
\end{proof}

\subsection{Filling area of curves} \label{subsec:filcurve}
Let $\gamma:S^1\to X$ be a continuous curve in  a complete metric space $X$. We define the \emph{filling area of $\gamma$ in $X$ with respect to $\mu$} as
$$\Fill (\gamma )=\Fill _{X,\mu} (\gamma ) :=\inf \{ Area_{\mu} (u)  \; |  \; u \in W^{1,2} (D,X)  \;  ,\;tr (u)=\gamma \}.$$

We let $I$ denote the unit interval and  $A$ denote the annulus $S^1\times I$. Any Sobolev map
$v\in W^{1,2} (A,X)$ has as a trace a map defined on the boundary of $A$ which consists of two copies of $S^1$.
Thus, the  trace of $v$ consists of two maps $\gamma _{0,1} \in L^2(S^1,X)$ and we say that $u$ is a Sobolev annulus connecting
 $\gamma _0$ with $\gamma _1$. Assume that  the Sobolev annulus $v$ connects two continuous curves $\gamma _0$ and $\gamma _1$. Then
gluing $v$ to a disc $u$ arising in the definition of the filling area  and reparametrizing the arising map   we deduce
$|\Fill (\gamma _0)-\Fill (\gamma _1)| \leq \Area_{\mu} (v)$, cf. \cite{LW-harmonic}, Section 3.2.  Any continuous curve $\gamma \in W^{1,2} (S^1,X)$
can be connected by a Sobolev annulus contained in the image of $\gamma$ to a constant speed parametrization $\gamma _0$  of $\gamma$.
Thus, $\Fill (\gamma )=\Fill (\gamma _0)$, \cite{LW-harmonic}, Lemma 2.6.

If $\gamma:S^1\to X$ is any rectifiable curve  with constant speed parametrization $\gamma _0$, then $\Fill (\gamma)$ might be infinite, while $\Fill (\gamma _0)$ is finite. Even if $X=\R^2$ there exist \emph{absolutely continuous} curves $\gamma:S^1\to \R^2$ which do not bound any Sobolev disc at all, since they do not belong to the fractional Sobolev space $W^{\frac 1 2,2} (S^1,\R^2)$, cf. \cite{Leoni}, \cite{Chiron}.

Finally,  $\Fill_{X,\mu} (\gamma) \leq Cl^2$  holds for any Lipschitz curve  $\gamma:S^1\to X$ of length $l\leq l_0$ if and only if
$X$ admits
the $(C',l_0)$-isoperimetric inequality with respect to  $\mu$ for all $C'>C$.

\subsection{Filling irregular  curves}
We are going to prove
\begin{thm} \label{thmfil}
Let $X$ be a space with a $(C,l_0)$-isoperimetric inequality. Assume that
$\gamma :S^1\to X$ is a curve of length $l\leq l_0$ which is the trace of a map  $u\in W^{1,2}(D,X)$.
Then $\Fill (\gamma ) \leq Cl^2$.
\end{thm}

This statement is not  obvious  even if $X$ is Euclidean space $\R ^n$. However, in this case it is a consequence of the following analytic fact, cf. \cite{Dierkes-et-al10}, p.283.
For the harmonic map $v:D\to \R^n$ with trace $\gamma$, the length of images of concentric circles is non-decreasing as a function of the radius.
Thus   one can find a small filling of $\gamma$  by taking a small annulus out of $v$ which connects $\gamma$ to a smooth curve $\eta$ of length not larger than $l$ and
then fill this smooth curve $\eta$. In general spaces $X$, it might very well happen, that all curves outside of $\gamma$ are much longer than $\gamma$.
Instead, we prove the following Lemma which immediately implies \tref{thmfil}:

\begin{lem} \label{lem:independent}
Let $X$ be  a space with a $(C,l_0)$-isoperimetric inequality. Let the rectifiable curve $\gamma:S^1\to X$ be the trace of a map  $u\in W^{1,2}(D,X)$.
Then there exist  Sobolev annuli $\hat u \in W^{1,2} (A,X)$ of arbitrary small  area which connect
$\gamma$ and its arclength parametrization $\gamma _0$. In particular, $\Fill_{X,\mu} (\gamma)=\Fill_{X,\mu} (\gamma _0)$. If $X$ is proper then the annuli $\hat u$
can be chosen  to be continuous and to satisfy Lusin's property (N).
\end{lem}

\begin{proof}
We may assume $u\in N^{1,2} (D,X)$.
We choose a very small $\delta$ which will control the area of $\hat u$.
Then we choose a small $\rho= \sin (\frac {2\pi} n)$ such that the restriction    of $u$ to the $\rho$-neighborhood of $S^1$ in $D$  has energy at most $\delta$.
We choose  equidistant points $p_1,....,p_n$ on $S^1$ with pairwise Euclidean distance $\rho$. If $\rho$ is small enough then the length of $\gamma$ on segments between consecutive points $p_i$, $p_{i+1}$ is smaller than $\delta$.

Denote by $E_i$ the energy of the restriction of $u$ to $B(p_i,\rho)\cap D$.
Use \lref{lem:ballcircle} to find  some  $r_i\in (\frac 2 3 \rho , \rho)$, such that the following holds true. The restriction of  $u$ to  the  distance circle $c_i$  of radius $r_i$ around
$p_i$ in $D$  is a continuous curve in $W^{1,2} (c_i,X)$ and the    length $l_i$ of $u\circ c_i$ satisfies $l_i^2<6\pi \cdot  E_i$.

We define $B_i$ to be the ball $B(p_i,r_i)$. By construction,   the boundary  $\eta $ of  $\cup B_i$ in $D$  is  a biLipschitz Jordan curve, and  the restriction of $u$ to $\eta$ is in $W^{1,2} (\eta,X)$.
Let $\Omega$ denote the annulus in $D$ between $\eta $ and $S^1$ and note that $\Omega$ is biLipschitz to $A$.
It suffices to find a map $v\in W^{1,2} (\Omega ,X)$ of area going to $0$ with $\delta$, whose traces coincide with $u|_{\eta}$ and $\gamma _0$ respectively.
Then after a
biLipschitz identification of $\Omega$ with $A$, we can glue both annuli $v$ and $u|_{\Omega}$ along $\eta$ to obtain the desired annulus $\hat u$ between $\gamma$ and $\gamma _0$.

In order to construct $v$ we proceed as follows.
The domain $\Omega$ is subdivided by the circular arcs  $c_i$ in $2n$ Lipschitz discs $T_j$. The boundary of any $T_j$
 consists of two or three
parts  of consecutive circles $c_i$ and a part $I_j$ of $S^1$. By our assumption on $\rho$, any restriction
$\gamma :I_j\to X$ has length  $m_j$ smaller than $2\delta$.  We now consider the  curve $k_j:\partial T_j \to X$
which coincides with $u$ on the parts of the circles $c_i$ and whose restriction to $I_j$ parametrizes the corresponding part of
$\gamma $ proportionally to arclength.   Thus $k_j$ is just a reparametrization of $u:\partial T_j \to X$ and has the same length
$a_j=\ell (k_j)$ as the restriction of $u$ to $\partial T_j$.

By our choice of the circles $c_i$, all the curves $k_j$ are Sobolev curves $k_j\in W^{1,2} (\partial T_j,X)$. Moreover, if $\rho$ has been chosen small enough, the length of any $k_j$ does not exceed $l_0$.
 \lref{no-rep}  provides  a map $v_j\in W^{1,2} (T_j,X)$ whose trace is $k_j$ and whose $\mu$-area is at most
$C\cdot a_j ^2$.  Whenever two domains $T_j$ and $T_l$ have a common part of the boundary, the traces of $v_j$ and $v_l$ on this common part coincide (with the restriction of $u$).   Therefore, gluing the $v_j$ together we obtain a well defined map  $v\in W^{1,2} (\Omega ,X)$ which coincide with $v_j$ on $T_j$ for all $j$. By construction, the trace of $v$ coincides with the restriction of $u$ on $\eta$. On the outer circle $S^1$, the trace of $v$ is a Lipschitz continuous reparametrization $\gamma _1$ of $\gamma_0$. Thus we can attach to $v$ another annulus of zero area connecting  $\gamma _1$ with $\gamma _0$, Subsection \ref{subsec:filcurve}, to obtain an annulus with   the required trace.  Therefore, it remains to prove that the $\mu$-area of the constructed annulus $v$  goes to $0$ as   $\delta$ goes to $0$.

 By construction, the $\mu$-area of $v$ is at most  $C\cdot \sum_{j=1}^{2n} a_j^2 $. The length $a_j$ of $k_j=u\circ \partial T_j$ has a contribution  $m_j$ of a part of the boundary circle $S^1$. The rest of $k_j$ consists  of two or three parts of circles $c_i$. We estimate those parts by the whole lengths $l_i$
 of the corresponding circular arcs $u\circ c_i$. Parts of each circle $c_i$ appear as boundaries of at most $5$ different domains $T_j$.  This implies $$\sum_{j=1}^{2n} a_j ^2 \leq K\cdot (\sum _{j=1}^{2n} m_j^2 + \sum _{i=1} ^n  l_i ^2),  $$
for some universal constant $K$.  Using the bound $l_i^2 < 6\pi \cdot E_i$ we obtain
$$\Area_{\mu} (v)\leq K_1\cdot (\sum _{j=1}^{2n}  m_j^2  +\sum _{i=1} ^n E_i),$$
for some constant $K_1$ depending on $C$.

Since the balls $B_i$ intersect at most pairwise, the sum  $\sum_{i=1}^n E_i$ is bounded by $2\cdot E_+^2 (u|_{\Omega} )\leq 2 \delta$.
 On the other hand $\sum _{j=1}^{2n}  m_j$ is the finite length $l$ of $\gamma$,
hence $\sum _{j=1} ^{2n}  m_j ^2$ becomes arbitrary small once $\delta $ has been chosen small enough.
 This finishes the proof that the constructed annulus $v$ has arbitrary small area, once $\delta$ is chosen small enough.

 If $X$ is proper, we may first replace $u$ by the harmonic map with the same trace, see \cite{LW-harmonic}. Then $u$ is continuous on $\bar D$ and satisfies  Lusin's property (N). Also the fillings $v_j$ can be chosen
to be continuous and satisfy Lusin's property (N), cf. \cite{LW-harmonic}.  Then the constructed annulus $v$ is continuous as well and satisfies Lusin's property (N).  Therefore, also the annulus $\hat u$
obtained from a gluing of $v$ and $u$ will be continuous and  satisfy Lusin's property (N).
\end{proof}

\section{Solutions of the Plateau problem} \label{sec:Plateau}
\subsection{Setting} Let $\mu$ be a definition of area. Let $X$ be a complete
metric space which admits a $(C,l_0)$-isoperimetric inequality for the area $\mu$.
Let $\Gamma$ be a rectifiable  Jordan curve  in $X$. Denote as always in this paper by $\Lambda (\Gamma, X)$ the set of maps
$u\in W^{1,2} (D,X)$, whose trace is a weakly monotone parametrization of $\Gamma$.  Let $u\in \Lambda (\Gamma,X)$  be a solution of the
Plateau problem for the curve $\Gamma$ with respect to  $\mu$. Thus  $\Area_{\mu} (u)$ is minimal in
$\Lambda (\Gamma ,X)$ and $u$ has minimal Reshetnyak energy $E_+^2 (u)$ among all minimizers of the $\mu$-area in $\Lambda (\Gamma, X)$.
 Due to  \cite{LW}, Theorem 9.1, $u$ has a unique representative which continuously extends to
$\bar D$.  From now on we will  fix this representative $u:\bar D\to X$.

The map  $u$ is $\sqrt 2$-quasiconformal  and it is conformal if $X$ satisfies property (ET), \lref{lem:Qq} and the subsequent paragraph. Due to \cite{LW}, Theorem 1.4, there exists a number $p>2$ depending only on
the constant $C$ such that $u$ is contained in $W^{1,p} _{loc} (D,X)$. Moreover,  if $\Gamma$ is
a chord-arc curve then one can improve the integrability globally: there is some  $p>2$ depending on
$C$ and the biLipschitz constant of $\Gamma$ such that $u\in W^{1,p} (D,X)$, \cite{LW-harmonic}, Theorem 1.3, see also  \cite{LW}, Theorem 1.4 (iii).
The last results and the Sobolev embedding theorems (\cite{LW}, Proposition 3.3) imply that $u$ is locally Hoelder continuous on $D$, respectively globally
Hoelder continuous on $\bar D$, with the  exponent  $1-\frac 2 p$.

\subsection{Restrictions to subdomains}
 The next result generalizes \pref{newprop}.
\begin{prop} \label{lip-good}
Let $u:\bar D\to X$ be a solution of the Plateau problem as above. Let $T\subset   \bar D$ be a Jordan curve with Jordan domain $O$.
  Assume  that $u|_T$ is a curve of finite  length $l$ and let $\gamma_0 :S^1\to X$ be a parametrization of
$u|_T$ proportional to arclength.  Then $\Area _{\mu} (u|_O) \leq \Fill _{X,\mu} (\gamma_0)$.  In particular, if $l\leq l_0$ then
$\Area _{\mu} (u|_O) \leq C l^2$.
 \end{prop}

\begin{proof}
Fix a conformal diffeomorphism $F:D\to O$.
Due to \cref{cor:minmin}, the map $v=u\circ F$ is contained in $N^{1,2} (D,X)$ and has minimal area among all maps with the same trace.
Since $F$ extends to a homeomorphism $F:\bar D\to \bar O$ by the theorem of Caratheodory, $v$ extends to a continuous map on $\bar D$.
Thus the trace of $v$ coincides with $u\circ F |_{S^1}$. Therefore, the trace of $v$ is a rectifiable curve of length $l$,  and $\gamma_0$ is
a reparametrization of this trace.
Due to  the minimality of $v$ and \cref{cor:minmin}, $\Area _{\mu} (v)= \Area _{\mu} (u|_O)$ is the filling area of the curve $u\circ F |_{S^1}$.
From \lref{lem:independent} we deduce  $\Area _{\mu} (u|_O) \leq \Fill _{X,\mu} (\gamma_0)$.
\end{proof}

\subsection{Intrinsic regularity} All regularity results in \cite{LW}  are based on the estimate of lengths  of some curves in the image of
$u$. This estimate goes back to  Morrey's classical proof of the a priori Hoelder continuity, \cite{Mor48},  and is proven \cite{LW}, Proposition 8.7, in the present setting of metric spaces with quadratic isoperimetric inequalities. As  shown  in \cite{LW-energy-area}, Theorem 4.5, the
Hoelder exponent  from \cite{LW}, Proposition 8.7, can be slightly improved using \lref{lem:Qq}.

\begin{lem} \label{shortcurve}
 For any  $1>\delta >0$ and $A>0$ there is  some $L=L(C,l_0,\delta ,A) >0$
such that  the following holds true   whenever the $\mu$-area of the $\mu$-minimal disc $u$ is at most $A$.
 For any $z_1,z_2\in B(0,\delta )$
there is a piecewise affine curve $\gamma$ from $z_1$ to $z_2$ such that $$\ell_X(u\circ  \gamma ) \leq  L\cdot |z_1-z_2| ^{\alpha},$$
where $\alpha  = q(\mu )\cdot \frac 1 {4\pi C}$.
\end{lem}

In fact, similarly to \cite{LW}, Proposition 8.7, the above result applies to the slightly more general situation where the $\mu$-minimality of $u$ is replaced by the following slightly weaker assumption: the map $u$ minimizes the $\mu$-area among all maps with the same trace as $u$, and $u$ satisfies the conclusion of \lref{lem:Qq}. This extension  is needed  only in the proof \lref{mainboundary} below.  In fact, the slightly smaller Hoelder exponent used in \cite{LW}, Proposition 8.7, suffices for the conclusion of \lref{mainboundary}.


\subsection{Boundary continuity}
For regular curves, the following result is implicitly contained in  the proof of the boundary regularity   in  \cite{LW}, Section 9.
We sketch  the proof, referring  to \cite{LW} for details.

\begin{lem} \label{mainboundary}
For any $\epsilon >0$ there exists some
$s >0$ depending on $C,l_0$ and $\epsilon$  with the following property.
Let $u\in \Lambda (\Gamma, X)$ be a solution of the Plateau  problem as above.
Let $T$ be a Jordan curve in $\bar D$  with Jordan domain $\Omega$.
    Assume that the restriction of $u$ to $\Omega$ has $\mu$-area at most $s$. Then for any $p\in \Omega$ there exists
	a curve $\eta \subset \Omega$ connecting $p$ with $T$  such that $\length _X (u\circ \eta) < \epsilon$.
\end{lem}

\begin{proof}
Let $0<s  <1$ be small enough, to be determined later.
Choose a conformal map $\phi:D\to \Omega$ such that $\phi (0)=p$. Denote by $\phi$ its continuous extension to a homeomorphism from $\bar D\to \bar \Omega$.
 As we have seen in \cref{cor:minmin}, the composition $\hat u:=u\circ \phi$ is contained in $W^{1,2} (D,X)$ and has  minimal $\mu$-area among all maps in
	$W^{1,2} (D,X)$ with the same trace as $\hat u$.  Since $\phi$ is conformal it preserves the $\mu$-area and the Reshetnyak energy on all subdomains.  Thus the composition
$\hat u=u\circ \phi:\bar D \to X$ satisfies the conclusions of \lref{lem:Qq}. As has been explained after \lref{shortcurve}, the conclusion of \lref{shortcurve} applies to the map $\hat u$.
 Thus, there exists some $1>\delta >0$  depending only on $C,l_0$ and $\epsilon$ with the following property. For any point $\theta \in S^1$ there exists some curve $\gamma _{\theta}$  in  $D $ connecting $0$ with $\delta \cdot   \theta $ such that $\ell_X(\hat u \circ \gamma _{\theta} )<\frac \epsilon 2$.

 The restriction of  $\hat u$ to the annulus  $\Omega= \{z \; | \; \delta  <|z|< 1\}$ is $\sqrt 2$-quasiconformal and  satisfies
 $E^2_+ (\hat u |_{\Omega}) \leq 2\Area _{\mu} (\hat u|_{\Omega}) \leq 2 s.$  Denote by  $ \eta  _{\theta} $ the radial curve connecting $\delta \cdot \theta $ and $\theta $ and by $\rho$ the minimal generalized gradient of $\hat u$.  Integrating in polar coordinates, using Hoelder inequality and \eqref{eq-n1p}  we deduce
$$2s \geq E_+^2 (\hat u |_{\Omega}) =\int _{\Omega} \rho ^2 (z) \, dz \geq \delta \int _{S^1} \left(\int _{\eta _{\theta}} \rho ^2 \right)\,  d\theta  \geq \delta \int _{S^1} \ell _X(\hat u \circ \eta _{\theta})^2. $$
Thus, if $s$ is small enough we find some $\theta \in S^1$ with $\ell_X(\hat u \circ \eta _{\theta} )<\frac \epsilon 2$.

  Now the concatenation $\gamma = \gamma _{\theta} * \eta  _{\theta}$ connects  the origin $0$ with $\theta \in \partial D$ and $\epsilon > \ell_X(\hat u \circ \gamma )$.
	Thus we obtain the required curve $\eta$ as $\eta =\phi \circ \gamma$.
\end{proof}

We can now deduce:

\begin{cor} \label{shortconnection}
Let $u \in \Lambda (\Gamma, X)$ be a solution of the Plateau problem as above.
Then
for any $\epsilon >0$ there exists some $ \epsilon > r >0$  with the following property. For any pair
of points $x,y\in \bar D$ with $|x-y|<r$ there is some curve $\gamma $ connecting these points inside $B(x,\epsilon)$ and such
that $\length _X (u\circ \gamma ) <\epsilon$.
\end{cor}

\begin{proof}
Using \lref{mainboundary} it suffices to prove the following claim.

\emph{Claim:} For any $\delta >0$ there exists some $r>0$ such that  any $x,y \in \bar D$ with $|x-y|<r$
 are contained in the closure $\bar \Omega$ of a convex domain $\Omega \subset D$  such that
 the  quantities $\Area _{\mu} (u|_{\Omega}), \diam (\Omega )$ and $\ell_X(u|_{\partial \Omega})$ are bounded from above by $\delta$.

 The claim is proven by taking $\overline {\Omega}$ to be the closure of the ball $B(x,t) \subset \bar D$ of an appropriate  radius $t\in (r,\sqrt r)$  around $x$.
 Indeed, if $r$ is small enough then the diameter of any  such ball is  certainly smaller than $\delta$. The boundary  $T:=\partial \Omega \subset \bar D$ of such balls consists  of the distance circle
 $S_t=S_t(x)$ and, possibly,
 an interval $T^b =\overline {\Omega} \cap \partial D$ in the boundary circle $\partial D$.  Since $u:\partial D\to \Gamma$ is rectifiable, the image
$u (T^b)$ has arbitrary small length, if the diameter of $T^b$ is small enough. On the other hand,  $\ell_X(u \circ S_t)$ is controlled by  the lemma of Courant-Lebesgue,
 \cite{LW}, Lemma 7.3, see also \lref{lem:ballcircle} above.
Therefore,  $\ell_X(u\circ T)$ becomes arbitrary small if $r$ is small enough and $t$ is chosen by    the lemma of Courant-Lebesgue.  Finally, an upper bound on $\Area_ {\mu} (u|_{\Omega})$ follows
 from the upper bound on $\ell_X(u\circ T)$  and \pref{lip-good}.
 \end{proof}

\section{Metric and analytic properties of the space $Z$} \label{sec:metric}
\subsection{Setting, notations and basic properties}  Throughout  this  section, we fix a definition of area $\mu$ and constants  $C, l_0>0$. We fix  a complete  metric space $X$ admitting a $(C, l_0)$-isoperimetric inequality.  We fix   a Jordan curve $\Gamma$ of finite length in $X$.  Finally, we let  $u:\bar D\to X$ be a solution of the Plateau problem for $(\Gamma,X)$.
Consider the pseudo-distance $d_u:\bar D\times \bar D \to [0,\infty]$ defined as in the introduction 
 $$d_{u} (z_1,z_2):= \inf \{\length_X  (u\circ  \gamma ) \,|\, \gamma \subset \bar D\,  , \;\gamma \; \text{connects} \; z_1 \; \text{and} \; z_2 \}.$$

Now the statement of \cref{shortconnection} immediately implies:
\begin{lem}
For any $\epsilon >0$ there exists some $r>0$ such that for $x,y\in \bar D$ with $|x-y|<r$ we have
$d_u (x,y) \leq \epsilon$.
\end{lem}

As a direct consequence of the triangle inequality we deduce that
the pseudo-metric  $d_u:\bar D\times \bar D\to [0,\infty]$ is finite-valued and  continuous.
  Consider the equivalence relation of $\bar D$ identifying pairs of points   $z_1,z_2\in \bar D$ with $d_u(z_1,z_2)=0$.
 Let $Z$ be the set of equivalence classes of points in $\bar D$. We obtain a canonical surjective projection $P:\bar D\to Z$.  The pseudo-metric $d_u$ defines a metric $d_Z$ on the set $Z$.
The continuity of $d_u$ implies that $P:\bar D\to Z$ is continuous. Thus $Z$ is a compact metric space. For any $z_1,z_2\in \bar D$  we have $d_u(z_1,z_2) \geq d_X(u(z_1),u(z_2))$.
Therefore, $u$ admits a unique factorization
$u=\bar u \circ P$  with a unique $\bar u:Z\to X$. Moreover,  $\bar u$ is $1$-Lipschitz.  Hence we have  already verified most statements in the following direct generalization of \tref{prop:first}.

\begin{prop} \label{prop:first+}
The pseudo-distance $d_u$ assumes only finite values  and is continuous. The metric space  $Z$ associated with the pseudo-metric $d_u$
is compact and geodesic, and the canonical projection $P:\bar D\to Z$ is continuous.
 The map $u:\bar D\to X $ has a canonical factorization  $u=\bar u\circ P$, where
$\bar u:Z\to X$ is a $1$-Lipschitz map. For any curve $\gamma :I\to \bar D$ we have
$\ell_X(u\circ \gamma )=\ell_Z(P\circ \gamma )$, hence $\bar u$ preserves the length of $P\circ \gamma$.
\end{prop}

\begin{proof}
It remains to prove the last equality and the fact that $Z$ is a geodesic space.  Thus let $\gamma :I\to \bar D$ be given.
Since $u=\bar u \circ P$ and $\bar u$ is $1$-Lipschitz, we have $\ell_X(u\circ \gamma ) \leq \ell_Z(P\circ \gamma )$. On the other hand, for any $[t,t']\subset I$
we have $d_Z(P\circ \gamma (t), P\circ \gamma (t') )\leq  \ell _X (u\circ \gamma |_{[t,t']})$. Thus the reverse inequality follows directly from the definition of length  \eqref{eq:rect}.

For any $p_1 =P(z_1),p_2 =P(z_2) \in Z$  the definition of $d_Z(p_1,p_2)=d_u (z_1,z_2)$ together with the equality of lengths proved above shows that $Z$ is a length space.
Since $Z$ is compact, the theorem of Hopf-Rinow shows that $Z$ is geodesic.
\end{proof}

Also the following result is general  non-sense as  well.

 \begin{lem} \label{lem:nonsense-P}
Any fiber $P^{-1} (q)$ is a connected subset of $\bar D$.
\end{lem}

\begin{proof}
The set $K=P^{-1} (q)$ is closed, hence compact. 	 If it is not connected we find a decomposition
$K=K_1\cup K_2$ such that $d(K_1,K_2)>0$. Let $S$ be the compact set of all points in $\bar D$  which are at the same distance from $K_1$ and $K_2$.
Choose $k_1\in K_1$ and $k_2\in K_2$. By continuity and compactness, the function
$x\mapsto d_u (k_1,x)$ assumes a minimum $\epsilon$ on $S$. Since $S$ does not intersect $K$, we have $\epsilon >0$.  By definition of $d_u$, we find a curve $\gamma $ connecting $k_1$ and $k_2$  with $\ell_X(u\circ \gamma )<\epsilon$.  This curve must intersect the set $S$
at some point $p$. We deduce $d_u(p,k_1) \leq \ell_X(u\circ \gamma ) <\epsilon$.
This contradiction finishes the proof.
\end{proof}

Since $u:S^1\to \Gamma$ is rectifiable the curve $P:S^1\to Z$ is rectifiable as well. Since the restrictions of $u$ and $P$ to any subarc of $S^1$ have equal lengths and since
$u|_{S^1}$ a weakly monotone parametrization of the   Jordan curve, $\Gamma$ we conclude:
\begin{lem} \label{lem:bcurve}
The restriction $P:S^1\to Z$ is a weakly monotone parametrization of a  rectifiable Jordan curve $\Gamma'$.    The restriction  $\bar u:\Gamma ' \to \Gamma$ is an arclength-preserving
homeomorphism.
\end{lem}

\subsection{Analytic properties}
From \lref{shortcurve} we infer:
\begin{lem} \label{h-optimal}
The restriction $P:D\to Z$ is locally $\alpha$-Hoelder continuous with $\alpha =q(\mu) \cdot \frac 1 {4\pi C}$.
\end{lem}

Since   $\ell_X(u\circ \gamma)=\ell_Z(P\circ \gamma )$ for all curves $\gamma$ in $\bar D$    we deduce
from \cref{twomaps}  and the corresponding property of $u$:

\begin{lem} \label{diff-equal}
The map $P:D\to  Z$ is in the Sobolev class $W^{1,2} (D, Z)$ and in the local Sobolev class $W^{1,p} _{loc}$ for some
$p>2$ depending on $C$.
The approximate metric differentials of $P$ and $u$ coincide at almost all points  $z\in D$.
In particular, the restrictions of $u$ and $P$ to any subdomain $O\subset D$ have equal $\mu$-area and equal energy.
\end{lem}

From  \lref{diff-equal} and the infinitesimal properties of $u$ we get:
\begin{cor} \label{P-quasi}
The map $P:D\to Z$ is $\sqrt 2$-quasiconformal. If the space $X$ satisfies  property (ET) then $P:D\to Z$ is conformal.
\end{cor}

From \pref{lip-good}, \lref{diff-equal} and the last statement in \pref{prop:first+} we directly deduce:
\begin{lem}  \label{isop-adm}
Let $T\subset \bar D$ be a   Jordan curve with Jordan domain $\Omega$.  If
$\ell_Z (P|_T ) \leq l_0$ then
$Area _{\mu} (P|_{\Omega}) \leq C\cdot \ell_Z (P|_T ) ^2.$
\end{lem}


Since $P\in W^{1,p} _{loc} (D, Z)$, for some $p>2$,
the map $P$ satsifies Lusin's property (N).  Thus, $P(D)$ is countably $2$-rectifiable and has
finite $\mathcal H^2 $-area,   Subsection \ref{subsec:areasob}.
  Since $Z=P(D)\cup\Gamma' $ and $\mathcal H^2 (\Gamma') =0$ we obtain:

\begin{lem} \label{2-rect}
The space $Z$ is countably $2$-rectifable and $\mathcal H^2 (Z)<\infty $.
\end{lem}

We consider the  multiplicity  function $N:Z \to [1,\infty]$ defined by $N(z)=\#\{x\in \bar D: P(x) = z\}$ which appears in the area formula. Due to \lref{lem:nonsense-P}, the   fibers of the map $P$ are connected, thus the function   $N$ can only assume  the values $1$ and $\infty$.  From \lref{area-sob} we deduce that
for $\mathcal H^2$-almost all point $z\in Z$ the value  $N(z)$ is  exactly $1$. Another application of the area formula in  \lref{area-sob}
now gives us:
\begin{lem} \label{simplearea}
For any open subset $ V\subset Z$ we have
$$\mu _Z (V)= \Area_{\mu} (P|_{P^{-1} (V) \cap D}).$$
\end{lem}

The measure $\mu_Z$ and hence $\mathcal H^2$  has the whole set $Z$ as its support:

\bl\label{lem:Haus-positive}
 For every $z_0\in Z$ and every $r>0$ we have $$\mu_Z(B(z_0, r))>0.$$
\el

\begin{proof}
 Otherwise, we find some $z_0=P(x_0)$ and some $r>0$
 with  $\mu _Z  (B(z_0, r))=0$.  The set   $\Omega:= P^{-1}(B(z_0,r))$
 is open in $\bar D$. It  consists of all points  which can be connected to $x_0$ by some  curve  $\gamma$ with $\ell_X(u\circ \gamma ) <r$. Therefore, $\Omega$ is connected.  Hence $\Omega \cap D=\Omega \setminus \partial D$ is connected as well.
  From \lref{simplearea} we deduce $\Area_{\mu}(P|_{\Omega \cap D})=0$. Since $P$ is quasi-conformal, the restriction of $P$ to $\Omega \cap D$ has vanishing energy.  Since $\Omega \cap D$ is connected and $P$ continuous we infer that $P$ is constant on $\Omega \cap D$. By continuity, $B(z_0,r)=P(\Omega)=\{z_0\}$. Since $Z$ is geodesic this implies that $Z=\{z_0\}$.  This is impossible since $\Gamma'=P(\Gamma )$
  is  a Jordan curve in $Z$.
\end{proof}

\section{Topological preliminaries} \label{sec:topo}
In this section we collect some well-known statements in $2$-dimensional topology and provide minor variants
of these statements. A reader with some experience in  this area  may proceed directly to the next section.

\subsection{Jordan's curve theorem} \label{subsec:jordan}
By the theorem of Jordan, any Jordan curve $\Gamma \subset S^2$
divides $S^2$ into two domains. These domains are (homeomorphic to) open discs with boundary $\Gamma$ and their closures are homeomorphic to $\bar D$.

Due to a theorem of Rado, these domains depend "continuously" on the Jordan curve in the following sense,  \cite{Pommerenke}, Theorem 2.11. Let
a Jordan curve $\Gamma$ be fixed, let $O$ be one of the corresponding domains and
$p\in O$ an arbitrary point.  Then for any $\epsilon >0$ there is some $\delta >0$
with the following property. If $f:\Gamma \to S^2$ is a homeomorphism onto the image  Jordan curve $\Gamma '$ such that $d(f(x),x) <\delta $  for all $x\in \Gamma$ then there exists  a  homeomorphism $F:  \bar O \to \bar O'$  with $d(F(x),x) <\epsilon$ for all
$x\in \bar O$.  Here we denote by $O'$ the Jordan domain of $\Gamma '$ which contains the point $p$.

Given disjoint subsets $A,B,C$ of a topological space $Y$, we say that $A$  \emph{separates} $B$ from $C$ if any connected subset $S$ of $Y$ which contains points of $B$ and $C$ must intersect $A$.
 A \emph{Peano continuum} is a compact,  connected, locally connected metric space. Any Peano continuum is locally path connected.
We will  need the following simple observation.

\begin{lem} \label{separate2sphere}
Let $Y$ be a simply connected Peano continuum.
Let $K$ be a compact subset  of $Y$ which separates two points $x$ and $y$ in $Y$. Then $K$ contains a minimal compact subset $K'$  which still  separates $x$ and $y$. Moreover, $K'$ is connected.
\end{lem}

\begin{proof}
If $K_j$ is a chain of decreasing compact subsets separating $x$ and $y$  then their intersection separates $x$ and $y$ as well.  Indeed, any $K_j$ intersects any curve connecting $x$ and $y$, hence so does $K'=\cap K_j$ by compactness.

 By Zorn's lemma there exists a minimal  compact subset $K'$ of $K$ which separates $x$ and $y$.
If $K'$ is not connected then it can be written as the non-trivial disjoint union of compact subsets $K'=K_1\cup K_2$.  Since $Y$ is simply connected, it has trivial first homology group $H_1 (Y,\mathbb Z)$.
The exactness of the  Mayer-Vietoris  sequence  in homology implies  the injectivity of the canonical map
 $$ H_0 (Y\setminus K' ,\mathbb Z) \to H_0 (Y\setminus K_1 ,\mathbb Z) \oplus  H_0 (Y\setminus K_2 ,\mathbb Z). $$

  Therefore, if points $x$ and $y$ define the same element in $H_0(Y\setminus K_1, \mathbb Z)$ and in
$H_0 (Y\setminus K_2 ,\mathbb Z)$ then they define the same element in $H_0(Y\setminus K' ,\mathbb Z)$, hence
are in the same component of $Y\setminus K'$.
Thus, either $K_1$ or $K_2$ must separate $x$ and $y$, in contradiction  with
the minimality of $K'$.
\end{proof}

We cite the following result from  \cite{Wil49}, Theorem IV.6.7:

\begin{lem} \label{peano-jor}
Let $K$ be a Peano continuum in  $S^2$ which separates two points $x$ and $y$. Then $K$ contains a Jordan curve  which still  separates $x$ and $y$.
\end{lem}



We will need:
\begin{lem} \label{open-compl}
Let $K\subset S^2$ be closed and connected. Then any component $U$ of the complement $S^2\setminus K$ is homeomorphic to a  disc.
\end{lem}

\begin{proof}
Otherwise we find a Jordan curve $T$ in $U$ which is non-contractible in $U$. Then
$K$ must contain at least one point in both Jordan domains defined by  $T$ in $S^2$.  Then $K$ cannot be connected.
\end{proof}

For the proof of the following result we refer to
\cite[Corollary 2B]{Fre90}:

\begin{lem} \label{peano}
Let $K$ be a  compact, connected metric space with finite $\mathcal H^1 (K)$.  Then
$K$ is a Peano continuum.
\end{lem}

Taking \lref{peano} and  \lref{peano-jor}  we arrive at:
\begin{cor}  \label{corh1}
Let $Y$ be a compact metric space homeomorphic to $S^2$. Let $K\subset Y$ be a compact subset
which separates  two points $x,y\in Y$. If $l=\mathcal H^1 (K) <\infty$ then $K$ contains   a rectifiable Jordan curve  of length  at most $ l$ which still separates
$x$ and $y$.
\end{cor}

  We can  now  deduce a corresponding separating result in discs:
  \begin{cor} \label{corh2}
  Let $Z$ be a metric space homeomorphic to $\bar D$. Let $K$ be a compact subset of  $Z$ with finite   $\mathcal H^1 (K)$.
  If $K$ separates points
$x,y \in Z$ then
$K$ contains a  minimal compact subset $T$ separating $x$ and $y$. Either $T$ is a  Jordan curve which intersects $\partial Z$ in at most one point. Or $T$ is  homeomorphic to a compact interval which intersects $\partial Z$ exactly at its endpoints.
  \end{cor}

\begin{proof}
We may assume without loss of generality that $x,y\not\in \partial Z$.
Due to  \lref{separate2sphere} we find a minimal compact   $T\subset K$ separating $x$ and $y$,  and $T$ is connected.
The set $T\setminus \partial Z$ separates $x$ and $y$ in
the open disc $Z\setminus \partial Z$. For any  subset  $T_0$ of $T\setminus \partial Z$ which separates $x$ and $y$ in  $Z\setminus \partial Z$, the closure $\bar T_0$ of $T_0$ in $Z$ separates $x$ and $y$ in $Z$, hence it coincides with
$T$ by minimality. Thus $T\setminus \partial Z$ does not contain any proper closed subset separating $x$ and $y$ in  $Z\setminus \partial Z$.
Using a Mayer-Vietoris sequence as in the proof of \lref{separate2sphere}
we deduce that $T\setminus \partial Z$ is connected.  Summarizing, we see that $T\setminus \partial Z$ is connected and dense in $T$.

Using  \lref{peano} we see that  $T$ is a Peano continuum.
Consider the doubling $Y$ of $Z$ along $\partial Z$ and let $T^+$ be the union of $T$ and $\partial Z$.  The compact set $T^+$ separates $x$ and $y$ in $Y$.
If $T$ does not intersect $\partial Z$ then $T$ and $\partial Z$ are connected components of $T^+$.  Since $\partial Z$ does not separate $x$ and $y$, we deduce that $T$
separates $x$ and $y$ in $Y$.  Due to \cref{corh1} and minimality $T$ is a Jordan curve in this  case.

Assume from now on that $T$ intersects $\partial Z$. Then $T^+$ is a Peano continuum, as a connected union of the  Peano
continua $T$ and $\partial Z$. Due to \lref{peano-jor}, we find
a Jordan curve $T^- \subset T^+$ which still separates $x$ and
$y$ in $Y$.

The set $T^-\setminus \partial Z \subset T$ separates
$x$ and $y$ in the open disc $Z\setminus \partial Z$.  Due to the discussion at the beginning of the proof, $T^{-} \setminus \partial Z =T\setminus \partial Z$.  Thus the connected set $T\setminus \partial Z$ is
an open  subset of the Jordan curve $T^-$. We deduce that $T$ is a connected subset of the Jordan curve $T^-$. Moreover, either $T=T^-$ and $T$ intersects $\partial Z$ in exactly one point, or $T$ is homeomorphic to a compact  interval and $T\cap \partial Z$ consists of the two endpoints of the interval.
\end{proof}

By induction we can derive the following extension of  \cref{corh2} to  sets of finitely many points:
\begin{lem} \label{lem:finiteset}
Let $Z$ be a metric space homeomorphic to $\bar D$. Let $F$ be a finite set $F=\{p_1,....,p_m \} \subset Z\setminus \partial Z$.
Let $K\subset Z \setminus F$ be a compact subset which separates
any pair of points of $F$. Then $K$ contains a minimal compact subset
$K_0$ which separates any pair of points of $F$ in $Z$.
If $\mathcal H^1(K) <\infty$  then
$K_0 \cup \partial Z$ is homeomorphic to a finite graph.  Moreover, $Z\setminus (K_0\cup \partial Z)$ has exactly $m$ connected components.
\end{lem}

\begin{proof}
The existence  of a minimal set $K_0$ follows as in the case of two points in \lref{separate2sphere}.   Thus we may assume that $K=K_0$  has finite $\mathcal H^1$-measure  and need to prove that $K\cup \partial Z$ is a finite graph, whose complement has exactly $m$ components.

We proceed by induction on $m$. If $m=2$, then the claim follows from \cref{corh2}.
Assume that $m> 2$  and the result is true for all $m'<m$.  For any pair of distinct  points
$p_i,p_j \in F$ the set $K$ contains a minimal compact subset $K_{ij}$ separating $p_i$ from $p_j$.
By minimality of $K$, we get $K=\cup _{1\leq i <j \leq m}  K_{ij}$.

Assume that some of the sets $K_{ij}$ intersect $\partial Z$.
Then any  such $K_{ij}$ is a simple arc or a Jordan curve and it  divides
 $Z$ into two closed discs $Y^{\pm}$ with common boundary (as subsets of $Z$) given by $K_{ij}$. Then we can  apply the inductive hypothesis to the intersection of $F$ and $K$ with those discs.
 This implies that $Y^{\pm} \setminus (K\cup \partial Y^{\pm}) $ has as many components as points in $F\cap Y^{\pm}$. Moreover, the union of $\partial Y^{\pm} $ and $Y^{\pm} \cap K$ is a finite graph. It follows that $K$ is a finite graph as well, and that $Z\setminus (K\cup \partial Z)$ has exactly $m$ connected components.

If, on the other hand, none of the sets $K_{ij}$ intersects $\partial Z$, then any $K_{ij}$ is a Jordan curve and $K$ is disjoint from $\partial Z$. We embed $Z$ into its double $Y$ homeomorphic to $S^2$. The sphere  $Y$ is divided by the Jordan curve  $K_{12}$ into two closed discs $Y^{\pm}$
 and we can  apply the inductive hypothesis to $Y^{\pm}$, $F^{\pm} =F\cap Y^{\pm}$ and $K^{\pm} :=K\cap Y^{\pm}$.  As above, we deduce that $K$ is a finite graph and that $Y\setminus K$  has exactly
 $m$ connected components, each of them containing exactly one point of $F$. Since $K$ does not intersect $\partial Z$, the union $K\cup \partial Z$ is again a finite graph.
 Moreover, the complement of $Z$ in $Y$ is an open  disc $O$ contained in one component $U$ of $Y\setminus K$. Then $U$ contains $\partial Z$ and
 $U\setminus O =U\cap Z$ is connected. We deduce that  $Z\setminus (K \cup \partial Z)$ has exactly $m$ components.
\end{proof}

\subsection{Cell-like maps}
The following definitions and statements can be found  in \cite[p.~97]{HNV04}, see also \cite{Edw78}.

\bd
 A compact space is called cell-like if it admits an embedding into the Hilbert cube $Q$ in which it is null-homotopic in every neighborhood of itself. A continuous surjection $f\colon X\to Y$ between metric spaces $X$ and $Y$  is called cell-like if $f^{-1}(q)$ is cell-like for every $q\in Y$.
\ed

Let $X$ and $Y$ be  compact metric spaces of finite topological dimension  and $f\colon X\to Y$ a continuous surjection. If $f\colon X\to Y$ is cell-like and $X$ is an absolute neighborhood retract then so is $Y$. If $X$ and $Y$ are absolute neighborhood retracts then $f$ is cell-like if and only if for every open set $U\subset Y$ the restriction  $f\colon f^{-1} (U)\to U$ is a homotopy equivalence.

Basic examples of cell-like sets are contractible sets. Any cell-like subset of $S^1$ is a closed interval.  In $S^2$ the situation is slightly more complicated but  it is still very
well understood:

\begin{example}\label{ex:subset-S2-cell-like}
 A compact subset $K\subset S^2$ is cell-like if and only if $K$ and  $S^2\setminus K$ are connected.
\end{example}

The most important class  of cell-like maps between absolute neighborhood retracts is given by uniform limits of homeomorphisms.  Sometimes, all cell-like maps are of this type. The next example is a direct consequence of the above characterizations of cell-like subsets of $S^1$:

\begin{example}\label{ex:simple-curve-cell-like}
 Let $Y$ be a compact metric space. A continuous surjection
$f\colon S^1\to Y$ is cell-like if and only if $Y$ is homeomorphic to $S^1$ and $f$ is a weakly monotone parametrization of $Y$.
\end{example}

In the $2$-dimensional case the corresponding result is a milestone in classical geometric topology and goes back  to Moore.

\bt\label{thm:Moore-cell-like}
 Let $X$ be a compact $2$-dimensional manifold without boundary and let $f\colon X\to Y$ be a cell-like map. Then $Y$ is homeomorphic to $X$ and $f$
is a uniform limit of homeomorphisms.
\et

In order to recognize the topology of our minimal disc we will need a  similar result for
manifolds with boundary. Unfortunately, we could not find a reference and, therefore, provide the proof of the following consequence of \tref{thm:Moore-cell-like}:

\bc\label{cor:Moore-disc}
 Let $Z$ be a compact metric space and $\varphi: \bar D\to Z$ a cell-like map.
Then $Z$ is a contractible and locally contractible space which is homeomorphic to the complement of some open topological disc $O$ in $S^2$. As a subset of $S^2$, the  boundary  of $Z$ is exactly $\varphi(S^1)$.

 If the restriction $\varphi|_{S^1}\colon S^1 \to \varphi(S^1)$ is  cell-like then
$\varphi$ is a uniform limit of homeomorphisms $\varphi_i:\bar D\to Z$.  In particular,
 $Z$ is homeomorphic to $\bar D$ in this case.
\ec

\begin{proof}
 Denote by $Y$ the space obtained by attaching a copy $\bar D'$ of $\bar D$ to $Z$ along the map $\varphi|_{S^1}$. Denote by $\iota\colon\bar D'\to Y$ the natural projection. View $S^2$ as the union of $\bar D$ and $\bar D '$, glued along $S^1$. Define $f\colon S^2\to Y$ by $f= \varphi$ on $\bar D$ and $f=\iota$ on $\bar D'$. Then $f$ is well-defined and cell-like. Therefore, by Theorem~\ref{thm:Moore-cell-like}, the space $Y$ is homeomorphic to $S^2$ and $Z$ is homeomorphic to the complement of the image of  $D'$ in the sphere $S^2$. This shows that $Z$ is homeomorphic to the complement of some open topological disc in $Y$. Moreover, as a subset of $Y$, the boundary of $Z$ is $\varphi(S^1)$. Since   $\varphi$ is cell-like, $\bar D$ is a $2$-dimensional absolute retract and $Z$ has dimension at most $2$, it follows that $Z$ is an absolute neighborhood retract and, in particular, locally contractible. Moreover, since $\varphi$ is a homotopy equivalence it follows that $Z$ is contractible. This proves the first statement.

If $\varphi|_{S^1}$ is cell-like as a map to $\varphi(S^1)$ it follows that $\varphi(S^1)$ is a Jordan curve by Example~\ref{ex:simple-curve-cell-like}. Since $Y$ is homeomorphic to $S^2$ the Schoenflies theorem shows that $\varphi(S^1)$ divides $Y$ into two domains $\Omega_1$ and $\Omega _2$ such that  $\overline{\Omega}_1$ and $\overline{\Omega}_2$ are homeomorphic to $\bar D$. Clearly, one of the two domains is exactly $Z$, viewed as a subset of $Y$.  Thus $Z$ is homeomorphic to $\bar D$.

 Due to \tref{thm:Moore-cell-like} the map $f:S^2\to Y$ is a uniform limit of homeomorphisms $f_i:S^2\to Y$. Then $Z_i=f_i (\bar D)$  and $Z=f(\bar D)$ are closed discs in the sphere $Y=S^2$  and $f_i$ converges uniformly to $f$. To obtain homeomorphisms $\varphi _i:\bar D\to Z$ we just need to change $f_i$ by a homeomorphism $\psi_i:Z_i \to Z$ which is close to the identity.   But the existence of such $\psi _i$ follows from the theorem of Rado, mentioned in
Subsection \ref{subsec:jordan}.
 \end{proof}

\subsection{A curve cutting lemma} In order to find a suitable Jordan curve inside some non-injective curve we will use the following
observation (only) in the case of a punctured disc $Y$.

\bl\label{lem:subcurves-injective-biLip}
 Let $Y$ be a locally contractible metric space. Let the curve $\gamma\colon S^1\to Y$ be
 non-contractible in $Y$.
 Then  there exists a weakly monotone parametrization $\eta:S^1\to Y$ of a    Jordan curve $T \subset \gamma (S^1)$ which is  non-contractible in $Y$ and such that
for every continuous map $F:Y\to X$ to another metric space $X$ we have  $\ell_X(F \circ \eta)\leq \ell_X(F\circ \gamma )$.
\el

\begin{proof}
Consider the set $\mathcal G$ of all   curves $\eta :S^1\to Y$ with the following property.
If $\eta (t)\neq \gamma (t)$ for some $t\in S^1$ then  $t$ is an inner point of an interval on which $\eta $ is constant.
For any maximal interval, on which $\eta \in \mathcal G$ is constant, the boundary points of $I$ are mapped by $\gamma $ to the same point in $Y$.
The curve $\gamma$ is contained in $\mathcal G$. The family $\mathcal G$ of curves is equi-continuous: the modulus of continuity of $\gamma$ gives also a modulus of continuity
for any $\eta \in \mathcal G$.

Denote by $\mathcal G^+$ the set of all not-contractible curves $\eta \in \mathcal G$.
For any $\eta \in \mathcal G$, let $O(\eta) \subset S^1$ be the open set of points around which $\eta$ is locally constant.  We claim that there exists some $\eta _0 \in \mathcal G^+$ for which $O(\eta _0)$ is maximal among all $\{ O(\eta )| \eta \in \mathcal G^+ \}.$
Assume that $\eta _i \in \mathcal G^+$ is  a sequence,  such that $O(\eta _i) \subset O(\eta_{i+1})$ for all $i$. The curves $\eta _i$ are equicontinuous curves in the compact set $\gamma (S^1)\subset Y$.
By the theorem of Arzela-Ascoli we find a  subsequence $\eta _j$ converging uniformly to a curve $\eta _0 :S^1\to Y$.
Due to the local contractibility of $Y$, $\eta_j$ is homotopic to $\eta _0$ for $j$ large enough, thus $\eta _0$ is non-contractible. If $\eta _0 (t) \neq \gamma (t)$ then $t\in O(\eta _j)$ for all $j$ large enough. We deduce that $\eta _0 \in \mathcal G$ and that $O(\eta _0)$ contains all subsets $O(\eta _j)$. An application of Zorn's lemma finishes the proof of  the claim.

If $\eta _0$ is not a weakly monotone parametrization of a Jordan curve we find some $t,t' \in S^1$ such that $\eta _0 (t)=\eta _0 (t')$ but $\eta_0$ is not constant on  any of the two intervals
$I^{\pm}$ of $S^1$ bounded by $t,t'$. Let $\eta ^{\pm}$ be the curve that coincides with $\eta_0$ on $I^{\pm}$ and is constant on the complementary interval $I^{\mp}$.  By definition, $\eta^{\pm}$
are contained in $\mathcal G$ and their constancy sets are strictly larger that $O(\eta _0 )$.  By the maximality of $O(\eta _0)$ we deduce that $\eta^{\pm} \notin \mathcal G^+$. Thus
$\eta^{\pm}$ are contractible curves.
But up to a reparametrization, $\eta_0$ is a concatenation of $\eta^+$ and $\eta ^-$. Thus $\eta_0$ is contractible, in contradiction with $\eta _0 \in\mathcal G^+$.  This contradiction shows that $\eta_0$ is a  weakly monotone parametrization of a Jordan curve.

By the definition of length, the inequality   $\ell_X(F \circ \eta)\leq \ell_X(F\circ \gamma )$ holds true for any continuous map $F:Y\to X$ and any $\eta \in \mathcal G$.
 \end{proof}

\section{Topological and isoperimetric properties of $Z$}\label{sec:inner-structure-mindiscs}
We proceed  using the notation  from   Section \ref{sec:metric}.

\subsection{Topology}
With our topological preparations we are in position to describe the topology of $Z$.

\begin{thm} \label{prop:structure-Z}
The space $Z$  is homeomorphic to $\bar D$ and $P:\bar D\to Z$ is a uniform limit of homeomorphisms.
\end{thm}

\begin{proof}
The restriction  $P:\partial D\to \Gamma'\subset Z$ is weakly monotone  by
  \lref{lem:bcurve}, hence cell-like by Example \ref{ex:simple-curve-cell-like}. Due to Corollary~\ref{cor:Moore-disc} it suffices to prove that $P:\bar D\to Z$ is cell-like. Thus we need to prove that for any $q\in Z$ the preimage $K=P^{-1}(q)$
   is a cell-like subset of $\bar D$.

Due to \lref{lem:nonsense-P}, the set $K$ is connected. Consider $\bar D$ as the lower hemisphere of $S^2$.   Due to  Example~\ref{ex:subset-S2-cell-like} it is enough to prove that $S^2\setminus K$ is connected.   Assume otherwise. Then there exists at least one component $O$ of $S^2\setminus K$ which does not intersect the closed upper hemisphere,
hence
$O$ is contained in $D$. Due to \lref{open-compl}, $O$ is homeomorphic to a disc, since $K$ is connected.  By the Riemann mapping theorem we find a conformal diffeomorphism $F:D\to O$.
Due to \cref{cor:minmin} the composition $v=u\circ F$ is contained in $W^{1,2} (D,X)$ and has minimal $\mu$-area among all maps with the same trace as $v$.  We claim that $\trace v$ is a constant curve.
By construction, $u(K) $ is a single point $p=\bar u (q)$.  For any sequence $z_j\in D$ converging to $\partial D$ the points $F(z_j)$ subconverge to some point in $K$. Therefore,  the sequence $v(z_j)=u\circ F(z_j)$ converges to the point $p$.  This proves the claim.

The constant curve $\trace v :S^1\to X$ can be filled by the constant disc. By minimality of $\Area_{\mu} (v)$ we deduce that $v$ has zero area.  But $\Area_{\mu} (v)=\Area _{\mu} (u|_O)$ which is non-zero, since $u$ is quasi-conformal and $u$ is non-constant on $O$. This contradiction finishes the proof.
\end{proof}

\subsection{Isoperimetric inequality}  We can approximate arbitrary Jordan curves in $Z$ by $P$-images of Jordan curves in $\bar D$
and use \pref{lip-good} to
  control
the isoperimetric properties of $Z$:
\bt\label{prop:Z-isop}
 Every Jordan curve $T$ in $Z$ bounds a unique open disc $\Omega\subset Z$. Furthermore,  if $\ell _Z(T) <l_0$ then
  \begin{equation}\label{eq:isop-Z}
   \mu_Z(\Omega) \leq C \length_Z(T)^2.
  \end{equation}
\et

\begin{proof}
Existence and uniqueness of  $\Omega$ is a consequence of  the Jordan curve theorem and
\tref{prop:structure-Z}. Since $P$ is a cell-like map, $P :P^{-1} (O)\to O$ is a homotopy equivalence, for any open subset $O\subset Z$.  In  particular,
 $P^{-1} (\Omega) \subset D$ is contractible, hence an open disc.
In order to estimate the area  of $\Omega$   we fix a small  $\varepsilon>0$ with $\length_Z(T) + \varepsilon<l_0$. We fix some open disc  $U\subset\Omega$, such that $\overline{U}\subset\Omega$ is homeomorphic to $\bar D$ and
 \begin{equation*}
  \mu_Z(U) \geq \mu_Z(\Omega) - \varepsilon.
 \end{equation*}
 Set $V:= P^{-1}(U)$. Then $V$ is contractible, hence homeomorphic to $D$.

Fix a homeomorphism $\gamma\colon S^1\to T$.
  Choose $\delta>0$ so small that the open $\delta$-neighborhood $N(T, \delta)$ of $T$ in $Z$  does not intersect $U$ and such that every ball of radius $\delta$ based at a point of $T$ is contractible in $Z\setminus U$.
  Let $\{t_0, t_1,\dots, t_k, t_{k+1}=t_0\}$ be a partition of $S^1$ such that $$\gamma([t_i,t_{i+1}])\subset B_Z(\gamma(t_i), \delta/2)$$ for every $i$. Choose $x_i\in\bar D$ with $P(x_i) = \gamma(t_i)$.
  By the definition of the metric in $Z$, there exists a  curve $\gamma_i$ in $\bar D$ from $x_i$ to $x_{i+1}$ such that $$\length_Z(P\circ\gamma_i)< \min\left\{\frac{\delta}{2}, d_Z(\gamma(t_i), \gamma(t_{i+1}))+\frac{\varepsilon}{k+1}\right\}.$$
	It follows that $\gamma_i$ does not intersect $V$. Let $\tilde{\gamma}$ be the concatenation of the curves $\gamma_i$ for $i=0, \dots,  k$. Then $\tilde{\gamma}$ is a closed  curve and $$\length_Z(P\circ\tilde{\gamma}) < \length_Z(\gamma) + \varepsilon<l_0.$$
	Moreover, $\tilde{\gamma}$ does not intersect $V$ and $P\circ\tilde{\gamma}$ is homotopic to $\gamma$ in $Z\setminus U$. In particular, $\tilde{\gamma}$ is not null-homotopic in $\bar D\setminus V$. By Lemma~\ref{lem:subcurves-injective-biLip} there exists a non-contractible Jordan curve $\eta$ in   $\bar D\setminus V$ with $\length_Z(P\circ \eta)\leq \length_Z(P\circ\tilde{\gamma})$. It follows that the Jordan domain enclosed by $\eta$ in $\bar D$ contains $V$.
  Hence, Lemma~\ref{simplearea} and  Lemma~\ref{isop-adm}  imply
  $$ \mu_Z(\Omega) -\varepsilon \leq \mu_Z(U) = \Area_{\mu}(P|_V) \leq C\length_Z(P\circ\eta)^2 \leq C(\length_Z(\gamma)+\varepsilon)^2.$$
 Since $\varepsilon>0$ was arbitrary this yields \eqref{eq:isop-Z}.
 \end{proof}

\begin{rmrk} The proof of \tref{prop:Z-isop} shows the following  slightly stronger statement. Let $T$ be any Jordan curve of finite length $l$ in $Z$ with Jordan domain $\Omega \subset Z$.
Then for every $\epsilon >0$ there exists a Jordan curve $\eta :S^1\to \bar D$ such that $\ell _X(u\circ \eta ) < l +\epsilon$ and such that $\Fill _{X,\mu} (u\circ \eta) +\epsilon  \geq \mu _Z(\Omega)$.
Moreover, reparametrizing $\eta$ if needed and using \lref{lem:independent} we may assume that $u\circ \eta$ is Lipschitz continuous.
\end{rmrk}

\begin{rmrk}
All subsequent results of this section are derived only using  \tref{prop:Z-isop} and the co-area  inequality,  \lref{co-area-gen}. If we just assume that a metric space
$Z$ satisfies the conclusions of \tref{thmA} and do not assume that $Z$ is countably $2$-rectifiable, the proofs below remain valid once every
  reference to
 \lref{co-area-gen}  is replaced by a reference to \lref{lem:fed}. We obtain as conclusions all theorems below with $\mu _Z$ replaced by $\mathcal H^2$ and $q(\mu _Z)$  replaced by $q(\mathcal H^2)=\frac \pi 4$.
\end{rmrk}

\subsection{Area growth}
In the sequel we will denote the Jordan curve
$\Gamma'=P(S^1) \subset Z$ by $\partial Z$, since it is the boundary circle of the topological disc $Z$.
 It is well known that isoperimetric inequalities often imply  lower bounds on volume growth:

\bt\label{prop:area-growth-Z}
Let $z_0\in Z$. Then
\begin{equation}\label{eq:growth-hm2-Z}
 \mu_Z(B(z_0, r))\geq \min \left\{ q(\mu) ^2\cdot \frac{1}{4C} \cdot r^2 \; , Cl_0 ^2 \right\}
\end{equation}
for every $0\leq r < d_Z(z_0, \partial Z) .$
\et

\begin{proof}
Fix any $z_0\in Z\setminus \partial Z$.
Assume that  there exists some $r_0>0$, which we fix from now on, such that  \eqref{eq:growth-hm2-Z} does not hold.
We set $b(r) := \mu _Z( B(z_0,r))$ for $r \leq r_0$.
Then the ball $B(z_0,r_0)$ does not intersect $\partial Z$ and $b(r_0) <  Cl_0 ^2$. Therefore,
$b(r) < Cl_0 ^2$  for all $r\leq r_0$.
 Let $h\colon Z \to\R$ be the $1$-Lipschitz distance function defined by  $h(z):= d_Z(z, z_0)$.
	Applying the co-area inequality  (\lref{co-area-gen}) to the restriction $h: B(z_0, r) \to \R$ we deduce:
 \begin{equation}\label{eq:coarea-Z}
 b(r)   \geq q(\mu) \cdot  \int_{0} ^{r}\hm^1 (h^{-1} (s))\,ds
\end{equation}
for all $r\leq r_0$.
 In particular, the compact set $S_r=h^{-1} (r)$
  has finite $\mathcal H^1 (S_r)$ for almost every $0\leq r  \leq r_0$.
	
We denote by $F(r)$ the right hand side of \eqref{eq:coarea-Z}.
 Then the function $F:[0,r_0] \to \R$  is absolutely continuous   and
\begin{equation} \label{fr}
F'(r) = q(\mu) \cdot \mathcal H^1 (S_r)
\end{equation}
   for almost all $0<r \leq r_0$. We claim
\begin{equation}\label{eq:hm-ball-ineq}
b(r) \leq C \cdot \hm^1 (S_r)^2
\end{equation}
for all $r\leq r_0$. Indeed, fix an embedding of $Z$ into $S^2$.  The compact subset $S_r$ separates $z_0$  from every point $z\in Z$ with $d_Z(z,z_0)>r$, hence from any point $q$ on $\partial Z$ which we fix now.  On the other hand, the ball $B(z_0,r)$ is connected, since the metric on $Z$ is intrinsic. If
$\hm^1 (S_r) =\infty$ then \eqref{eq:hm-ball-ineq} is valid. On the other hand, if $\hm ^1 (S_r) $ is finite we can apply \cref{corh1} and find
a Jordan curve $T \subset S_r$  which still separates $z_0$ from $q$.
Then $p$ and therefore the whole ball $B(z_0,r)$ must be contained in the
Jordan domain of $T$. Since $\length  _Z(T) \leq \hm^1 (S_r)$
we deduce \eqref{eq:hm-ball-ineq} from \tref{prop:Z-isop}.

 Taking \eqref{eq:coarea-Z}, \eqref{eq:hm-ball-ineq} and \eqref{fr} we deduce for almost all
 $r\leq r_0$ the inequality
\begin{equation}\label{eq:F-growth}
 F(r) \leq  b(r) \leq  C \cdot  q(\mu) ^{-2} \cdot  [F'(r)]^2.
\end{equation}
  Due to Lemma~\ref{lem:Haus-positive}, we have   $F(r)>0$ for all $r>0$. Thus integrating \eqref{eq:F-growth} yields $b(r_0)\geq F(r_0)\geq q(\mu) ^2\cdot  \frac 1 {4C}  \cdot  r_0^2$,
   in contradiction with  our assumption.  This contradiction finishes the proof.
  \end{proof}

\subsection{Linear local contractibility}
The isoperimetric inequality and the lower bound on the area growth of balls imply uniform linear local contractibility of the space $Z$:
\begin{thm} \label{prop:loccontr}
Let $Z$ be as above. For  any $0<r<\frac {l_0} 2$ and any $z\in Z$,
 the ball $B(z,r)$ in $Z$ is contractible inside  the ball $B(z,(8C+1)\cdot r)$.
\end{thm}

\begin{proof}
	Fix a  point $z\in Z$ and consider the open connected sets $O=B(z,r)$ and $U=B(z,(8C+1)\cdot r)$, where $r<\frac {l_0} 2$.  	Assume that $O$ is not contractible in $U$.
Note that $O$ has trivial  higher homotopy groups as does any open subset of the disc. Thus  we find a curve $\gamma :S^1 \to O$ which is non-contractible in $U$. Since $O$ is locally contractible, we may replace parts of $\gamma$ of small
	diameter by short geodesic segments.  Thus we may assume that $\gamma$ is a concatenation of short geodesics $\gamma=\gamma _1*\gamma _2 *\dots *\gamma _k$.
	Connect the starting and end point of $\gamma _i$ by a geodesic
$\eta_i ^{\pm}$ with the point $z$. We may assume $\eta _{i+1} ^- =\eta _i ^+$ for all $i$.
 For any $1\leq i\leq k$, these geodesics $\eta _i ^{\pm} $ together with $\gamma _i$ provide a
	closed piecewise geodesic  curve $c_i$ of length smaller than  $2r$. Moreover, $\gamma $ is homotopic to the concatenation $c_1 * c_2 *\dots  *c_k$.
	Thus one of the curves $c_i$ is non-contractible in $U$. Hence we may assume   $\ell _Z (\gamma ) <2r<l_0$.
	Using Lemma \ref{lem:subcurves-injective-biLip} we may further assume that $\gamma $ is a Jordan curve.
	
	Consider the Jordan domain $\Omega $ of $\gamma$ in $Z$ and deduce from
	\tref{prop:Z-isop} that $\mu _Z (\Omega) < 4Cr^2$.   Since $\gamma$ is contractible in $\bar \Omega$, we find a point $y\in \Omega \setminus U$, hence
	$d_Z  (z,y) \geq (8C+1) r$.  Due to the triangle inequality, the distance  $d_Z(y, \gamma )$ from
$y$ to the curve $\gamma $ is larger than  $8Cr$.  The connected open ball $B(y,8Cr)$ does not intersect $\gamma$, hence it does not intersect $\partial Z$ and is completely contained
	in $\Omega$. In particular, its area is bounded from above by the area of $\Omega$, which is at most $4Cr^2 < Cl_0^2$. From \tref{prop:area-growth-Z} and using
 $q(\mu) \geq \frac 1 2 $,
 we deduce $\mu _Z(\Omega) \geq  q(\mu) ^2 \cdot \frac 1 {4C} \cdot (8Cr)^2 \geq 4Cr^2$.
	
This contradiction finishes the proof.
\end{proof}

 Using that $Z$ is a closed disc we directly deduce:
\begin{cor} \label{cor:diambound}
Let $V$ be a subset of $Z$ homeomorphic to a closed disc. If the boundary circle $\partial V$ of $V$ has diameter $r<\frac {l_0} 2$ then $V$ has diameter
at most $(8C+1)\cdot r$.
\end{cor}

Finally, note that the assumption $r<\frac {l_0} 2$ in \tref{prop:loccontr} was only needed in one step, namely to assure that the area of some domain in $Z$ does not exceed $Cl_0^2$, a condition which is automatically satisfied if $l=\partial Z<l_0$.  Thus no bound on $r$ is needed in this case. In particular, \cref{cor:diambound} can be applied to the whole disc $V=Z$, using that $\partial Z$ has diameter at most $\frac l 2$ in this case. Thus we have:
\begin{cor} \label{cor:loccontr}
If $\partial Z$ has length $l$ smaller than $l_0$ then $Z$ has diameter at most $(4C+\frac 1 2)\cdot l$. Moreover, no bounds on $r$ are needed in \tref{prop:loccontr} and \cref{cor:diambound}.
\end{cor}

\subsection{Equi-compactness}
The area growth of balls  implies that the number of disjoint balls of a given radius in $Z$ can be bounded in terms of the area of $Z$.
More precisely:

\begin{thm} \label{thm:equi}
 Let   $Z$ be as above. Set $A=\mu (Z)$ and $l=\ell _Z (\partial Z)$.
 For any integer $n >\frac{l}{8Cl_0}$ there exists
some $\frac l n$-dense subset $F_n$ in $Z$ with at most $2  n +   64 C \cdot \frac {A}  {l^2}  \cdot n^2$  elements.
\end{thm}

\begin{proof}
Let $U_n$ denote the $\frac l {2n}$-tubular neighborhood of $\partial Z$ in $Z$ and $V_n=Z\setminus U_n$.  Let $F_n^+$ denote a maximal $\frac l n$-separated subset in $V_n$ and let $F_n^-$ denote a maximal $\frac l {2n}$-separated subset in $\partial Z$. Then $F_n^+$ is $\frac l n$-dense in $V_n$ and  $F_n^-$ is $\frac l n$-dense in $U_n$.  Hence the union $F_n$ of $F_n^+$ and $F_n^-$ is $\frac l n$-dense in $Z$.  Since $\partial Z$ is a $1$-Lipschitz image of the circle of length $l$, the set  $F_n^-$ has at most $2n $ elements.

On the other hand, for any $p\in F_n^+$ the ball $B_p:=B(p ,\frac l {2n})$ does not intersect $\partial Z$. From \tref{prop:area-growth-Z}, the general estimate $q(\mu) \geq \frac 1 2$, and the assumption on $n$ we infer:
\begin{equation} \label{areabt}
\mu  (B_p) \geq  \min \left\{ \frac 1 {16 C} \cdot \frac {l^2} {4n ^2}  \; , \; Cl_0^2 \right\} = \frac{l^2}{64 C\cdot n^2}.
\end{equation}
Moreover, all these balls  $B_p$ are disjoint.   Thus,  the number  of elements in   $F_n^+$  times the right-hand side of \eqref{areabt}  is not larger than $A$.
Hence $F_n^+$ has  at most $64 C\cdot \frac {A } {l^2} \cdot n^2$ elements.
This finishes the proof.
\end{proof}

If $l<l_0$ then $\frac A {l^2} \leq C$ and, moreover, no bound on $n$ is needed
to conclude \eqref{areabt}. Therefore:

 \begin{cor} \label{cor:1000}
 Assume that the length $l$ of $\partial Z$ is smaller than  $l_0$. For any integer $n$ the set $Z$ contains some $\frac l n$-dense  subset $ F_n$ with at most
 $2n + 64 C^2 n^2$ elements.
 \end{cor}

\subsection{Decomposition of $Z$ by a graph}
The following result is a  topological version of a similar discrete statement proved in \cite{Papasoglu} for curves in groups with quadratic isoperimetric inequality.

\begin{thm} \label{thm:papa}
Let $Z$ be as above. There exists a constant $M$ depending on $l_0$ and  the upper bounds of $C$, the area  $A$ of $Z$, and length $l$ of $\partial Z$ such that the following holds true.
For any integer $n$ there exists  a finite connected graph $\partial Z\subset G \subset Z$, such that $Z\setminus G$ has at most
$ M\cdot  n^2 $ components, and such that  any of these components is a topological disc of diameter at most
$\frac l n$.
\end{thm}

\begin{proof}
It suffices to prove the result for all $n$ satisfying  $n\geq \frac{l}{8Cl_0} $ and $\frac {4l} n < \frac {l_0} 2 $.
	 Due to Theorem \ref{thm:equi}, we find  an $\frac l n$-dense  subset $F=F_n$ in $Z$ with elements $p_1,....,p_m$ for some  $m\leq M_1\cdot    n^2$ for some $M_1$ depending
	on the upper bounds of $C,l,A$.
	Since $Z\setminus \partial Z$ is
	 dense in $Z$, we may assume that $F$ is contained in the open disc $Z\setminus \partial Z$.
	
	The idea is now to consider the \emph{Voronoi domains} defined by the set $F$.  However, we need  a few minor modifications.
For all $1\leq i \leq m$, let $\hat f_i :Z\to \R$ be the distance function to the point $p_i$. Since the functions
$\hat f_i$ are $1$-Lipschitz, we deduce from \lref{co-area-gen} that for almost all
$\epsilon >0$ the fiber $(\hat f_i- \hat f_j)^{-1} (\epsilon)$ has  finite $\mathcal H^1$-measure, for any $1\leq i, j \leq m$.

Thus we find arbitrary small  $\epsilon >0$ such that $(\hat f_j- \hat f_i) ^{-1}  (k\epsilon )$ has finite $\mathcal H^1$-measure for any $1\leq i, j \leq m$
and any $1 \leq k\leq m$.  We fix such $\epsilon$ satisfying
$$4m\cdot \epsilon \leq \inf \{d(p_i,p_j) , j\neq i \} \leq \frac {2l} n.$$  Consider the modified distance functions $f_i:Z\to \R$ given by $f_i(z) :=d(p_i,z) +i\cdot \epsilon$.
By construction, for any $i\neq j$ the set $S_{ij}$ of points $z$ with $f_i(z)=f_j(z)$  has  finite $\mathcal H^1$-measure.

Let $U_i$ be the set of points $z\in Z$ with $f_i(z) <f_j(z)$ for all $j\neq i$.  The sets $U_i$ are open. By assumption on $\epsilon$, we have $p_i \in U_i$.
The functions $f_j$ are $1$-Lipschitz and  decrease with velocity $1$
on any geodesic connecting $z$ with $p_j$.  Therefore,  for any $z\in U_i$ any geodesic from $z$ to $p_i$ is entirely contained in $U_i$. In particular, $U_i$ is connected.
Since  $F$ is $\frac l n$-dense in $Z$, and due to the smallness of $\epsilon$, any point $z\in U_i$ has  distance at most $\frac {2l} n$ to $p_i$. Therefore,
the  diameter of $U_i$ is at most $\frac {4l} n$.

Denote by $K$ the union of all the sets $S_{ij}$. Then $K$ is a compact subset of $Z$, which has finite
$\mathcal H^1$-measure and separates points $p_i$ pairwise. The complement $Z\setminus K$ has exactly $m$ connected components  $U_i$.
Let $K_0$ denote a minimal compact subset of $K$ which still separates the points from $F$ pairwise.
We deduce from \lref{lem:finiteset} that $K_1=K_0 \cup \partial Z$  is a  finite graph, whose  complement $Z\setminus K_1$ has exactly  $m$ connected components $W_1,..., W_m$ containing the corresponding points $p_i$. Since $U_i$ is connected,  we obtain $U_i \subset W_i$.  Since $K_1$ is nowhere dense in
$Z$, the sets $U_i$ are dense in $W_i$ for all $i$. Thus  the diameter of any $W_i$ is also bounded
by $\frac {4l} n$.

 Let now $G$ be the connected component of $\partial Z$ in $K_1$. This is  a finite connected graph, which is open in $K_1$.
 Any component $V$ of $Z\setminus G$ must intersect at least one of the components $W_j$, hence $V$ must contain this component $W_j$ in this case.
We deduce that $Z\setminus G$ has at most $m$ components.

It remains to control the size  of these possibly larger components of  $Z\setminus G$.
 By \lref{open-compl}, any component $V$ of $Z\setminus G$ is homeomorphic to an open disc.
   We claim that the boundary $\partial V$ of this disc  has diameter at most $\frac {4l} n$.
    Indeed, some neighborhood $O$ of $\partial V$ in  $\bar V$ intersects $K_1$ only in $\partial V$.
    Hence, choosing such a connected neighborhood $O$, we deduce that $O\setminus \partial V$ is contained in one of the components $W_j$. Thus its diameter is bounded by $\frac {4l } n$. We deduce the same bound for $\bar O$, hence for $\partial V$. From \cref{cor:diambound} we infer that $V$
		 has diameter at most $(8C+1)\cdot \frac {4l} {n}$.

    We set  $N=\frac {n} {4\cdot (8C+1)} $. Then the diameter of any component of $Z\setminus G$ is at most $\frac l N$. Moreover, $Z\setminus G$ has at most $m \leq M_1 \cdot n^2 = M \cdot N^2$
 components for a constant $M$ depending only on $M_1$ and $C$, hence only on the upper bounds of $l,A$ and $C$.
\end{proof}

Note again that if the length $l$ of $\partial Z$ is smaller than $l_0$ then one does not need any additional assumption on $n$ in the first lines of the above proof. Recall  that
 $C\geq \frac 1 {8\pi}$, \cite{LW}, Corollary 1.6. Thus, we can estimate $4(8C+1)$ by $k\cdot C$ and $2n + 64 C^2\cdot n^2$ by $kC^2\cdot n^2$ for a universal constant $k$. Thus, from \cref{cor:1000} and the last lines of the proof of \tref{thm:papa}  we obtain:
\begin{cor} \label{cor:papa}
If the length $l$ of the boundary curve $\partial Z$ is smaller than $l_0$ then the constant $M$ in \tref{thm:papa} can be chosen to be $\hat M \cdot C^4$ for some universal constant $\hat M$.
\end{cor}

\section{Collecting the harvest} \label{sec:harvest}
We now provide the proofs of the main theorems stated in the introduction.
We formulate them for the general definition of areas $\mu$ and continue to use the notations of the previous sections.

\subsection{General case}
We begin with the following generalization of \tref{thmA}.
Recall that $\frac \pi 4$ in \tref{thmA} is equal to $q(\mathcal H^2)$.

\begin{thm} \label{thmA+}
The metric space  $Z$ is  homeomorphic to $\bar D$. It is countably $2$-rectifiable with finite $\mathcal H^2(Z)$.
 For any Jordan curve $\eta$  in $Z$ of length $l<l_0$,  the domain $\Omega$ of the  disc $Z$ enclosed by
    $\eta$ satisfies
    \begin{equation} \label{Z-hausdorff-isop+}
    \mu _Z (\Omega) \leq C \cdot l^2.
\end{equation}
\end{thm}

\begin{proof}
The space $Z$ is homeomorphic to $\bar D$ by \tref{prop:structure-Z}. It is countably $2$-rectifiable with finite $\mathcal H^2 (Z)$  by \lref{2-rect}. The   isoperimetric property
 \eqref{Z-hausdorff-isop+} is exactly \tref{prop:Z-isop}.
\end{proof}

The results stated in  \cref{cor:neu} have been all proven in the more general form in the previous Section. Namely, under the assumptions of \cref{cor:neu} the  $\mu=\mathcal H^2$-area
of $Z$ is less than the critical value $Cl_0^2$.  The area growth (i) of \cref{cor:neu} is exactly the statement of  \tref{prop:area-growth-Z}. The uniform local contractibility (ii) is contained in
\pref{prop:loccontr} and \cref{cor:loccontr}.
Finally, the decomposition statement (iii) of \cref{cor:neu} is contained in \tref{thm:papa} and \cref{cor:papa}.

The next theorem generalizes \tref{thmB}.

\begin{thm} \label{thmB+}
The canonical projection $P:\bar D\to Z$
is a uniform limit of homeomorphisms $P_i:\bar D\to Z$. Moreover,
\begin{enumerate}
\item  $P\in \Lambda (\partial Z, Z)
 \subset W^{1,2} (D,Z)$.
\item The map $P:D\to Z$ is contained  in $W^{1,p} _{loc} (D,Z)$  for some  some $p>2$
depending on $C$.
\item The map $P:D\to Z$ is locally $\alpha$-Hoelder with $\alpha = q(\mu)  \cdot \frac  1 {4\pi  C}$.
\item  The equality $\mu _Z (P (V))= \Area _{\mu} (P|_{V})=\Area _{\mu} (u|_{V})$ holds true for all
open subsets $V\subset D$.
\item The map $P$ is $\sqrt  2$-quasiconformal. If $X$ has property (ET) then $P$ is conformal.

\end{enumerate}
\end{thm}

\begin{proof}
The first statement was proved in \tref{prop:structure-Z}.
In \lref{diff-equal} we showed that $P\in W^{1,2} (D,Z)$  and (ii).
Since $P$ restricts to a weakly monotone parametrization $P:\partial D\to \partial Z$  by \lref{lem:bcurve}, we deduce $P\in \Lambda (\partial Z ,Z)$.
The statement of (iii) is contained in \lref{h-optimal}.
 The second equality of  (iv) is contained in \lref{diff-equal}. The first equality follows from  the definition of the $\mu$-area \eqref{eq:area-sob}and the fact
that the multiplicity function $N$ appearing in the area formula \lref{area-sob} equals $1$ almost everywhere on $Z$, as proven after \lref{2-rect}.
Statement (v) is contained in \cref{P-quasi}.
\end{proof}

Now we turn to the generalization of \tref{thmC}.

\begin{thm}  \label{thmC+}
For every  $\epsilon >0$ there exists a decomposition $Z=S\cup _{1\leq i<\infty} K_i$ with compact $K_i$ and
$\mu _Z(S)=0$ such that the restrictions $\bar u:K_i \to \bar u(K_i)$ of the $1$-Lipschitz map $\bar u$ are  $(1+\epsilon)$-biLipschitz. Moreover,  for any $1\leq i <\infty$ and any $x\in K_i$ we have
\begin{equation} \label{eq:iso}
\lim _{y\to x, y\in K_i} \frac {d_Z(x,y) } {d_X(\bar u (x) ,\bar u (y))} =1.
\end{equation}
  Moreover, $\bar u :\partial Z \to \Gamma$ is an arclength preserving homeomorphism.
\end{thm}

\begin{proof}
Since the map $P:\bar D\to Z$ satisfies Lusin's property (N), we combine  Subsection \ref{subsec:apmd} and Subsection \ref{subsec:rect} and obtain  a disjoint decomposition  $\bar D=S_0\cup _{1\leq i <\infty}L_i$ with the following properties. The subsets $L_i$ are compact and $\mathcal H^2(P(S_0))=0$.
 The restriction  $P:L_i\to P(L_i)$ is a  biLipschitz map, which has a metric differential at each point.  Moreover,  this metric differential is  a norm at each point $z\in L_i$  and coincides with the approximate metric differential $\apmd P _z$.  Finally, if one considers $L_i$ with the metric determined by any of these norms $\apmd u_z$ then $P:L_i\to P(L_i)$ is $(1+\epsilon)$-biLipschitz.

 Since $u=\bar u \circ P$,  the restriction  $u:L_i\to u(L_i)$ is a Lipschitz map. Due to  \lref{diff-equal} at almost all points of $L_i$ its metric differential coincides with $\apmd u _z=\apmd P_z$.  Decompose every  $L_i$ in a negligible set and countably many  compact sets on which $u$ is $(1+\epsilon)$-biLipschitz with respect to the appropriate norm.  Taking  all these negligible subsets together  into a set $S_1$
with $\mathcal H^2(S_1)=0$,  we obtain a decomposition  $\bar D= S_0 \cup S_1\cup _{1\leq j<\infty} M_j $, with $S_0$ from above, such that the following holds true. The sets $M_j$ are compact, the restrictions of $u$ and $P$ to $M_j$
have equal metric differentials at all points of $M_j$. Moreover, $u:M_j\to u(M_j)$ and $P:M_j\to P(M_j)$ are $(1+\epsilon)$-biLipschitz if $M_j$ is equipped with the norm $\apmd u _z$ for some $z\in M_j$.

Hence, for the compact set $K_j=P(M_j)$, the restriction $\bar u:K_j\to u(M_j)$ is $(1+\epsilon)^2$-biLipschitz. Moreover, since $P$ and $u$ have the same metric differentials at all $z\in M_j$, we get \eqref{eq:iso}.
We note that $\mu _Z (P(S_0))=0$ by construction and
$\mu_Z (P(S_1))=0$, since $P$ has Lusin's property (N).  Thus, with  $S=P(S_0\cup S_1)$ we have written
$Z=S\cup _{1\leq j<\infty}  K_j$, such that $\mu_Z (S)=0$ and the restriction of $\bar u$ to any $K_j$ has all the required properties.
It can happen that this union is not disjoint. Then we make it disjoint by a further subdivision, noting that with $K_j$ any compact subset of $K_j$ has the property required in the statement of the theorem.

The last statement is contained in \lref{lem:bcurve}.
\end{proof}

\subsection{The chord-arc case}
We are generalizing \tref{arc-chord} now.

\begin{thm}  \label{arc-chord+}
Assume in addition that $\Gamma$ is a chord-arc curve. Then $P\in W^{1, p} (D,Z)$
 for some $p >2$  depending on $C$ and the biLipschitz constant $L$ of some parametrization $S^1\to  \Gamma$.
  In particular, $P:\bar D\to Z$ is globally $(1-\frac 2 p)$- Hoelder continuous.

 There exists $\delta >0$, depending only on $C,l_0$ and $L$,  such that   for all
$z_0\in Z$ and all $0\leq r\leq \delta$ we have
\begin{equation} \label{volnoncol}
 \mu_Z(B(z_0, r))\geq \delta \cdot r^2.
\end{equation}
\end{thm}

  \begin{proof}
There exists some $p>2$ depending only on $C$ and  $L$ such that $u\in W^{1,p} (D,X)$, by \cite{LW-harmonic}, Theorem 3.1.
 Due to Lemma \ref{diff-equal}, we get
$P\in W^{1,p} (D,Z)$.

It remains to prove \eqref{volnoncol}. We fix  a sufficiently small $\delta$, to be determined later and
proceed in analogy with the proof of \tref{prop:area-growth-Z}. We may assume $ \delta  ^3 <Cl_0^2$.
Consider an arbitrary  $z_0\in Z$ and set $b(r) =\mu_Z(B(z_0, r))$. We consider  the distance function
 $h:Z\to \R$ from the point $z_0$ and the corresponding level sets $S_r$, the distance spheres around $z_0$.  Finally, we fix a point $q\in \partial Z$ with maximal distance on $\partial Z$ from $z_0$ and note that  $d_Z(z_0,q)>\delta$
 if $2\delta <\diam (\partial Z )$.

If for some $r_0<\delta$ we have $b(r_0) \geq Cl_0^2$ then \eqref{volnoncol} holds true for all $r_0\leq r\leq \delta$.
  Arguing as in the proof of \tref{prop:area-growth-Z},
 we only need to find some constant $k>0$, such  that for all $r<\delta$ with $b(r) <Cl_0^2$ the following inequality holds true:
\begin{equation} \label{last}
 k\cdot \mathcal H^1 (S_r) ^2 \geq  b(r).
 \end{equation}

 The inequality is trivially fulfilled if $\mathcal H^1 (S_r)  =\infty$.  For all $0<r\leq \delta$, the set $S_r$ separates $p$ from $q$. For all $r$   with finite $\mathcal H^1(S_r)$, we  apply   \cref{corh2} and  find a subset $\gamma$ of $S_r$ still separating $p$ from $q$ such that one of the following two possibilities holds true.
 Either $\gamma $ is a Jordan curve.  Then  the same argument as in the proof of \tref{prop:area-growth-Z}  gives us
 $C\cdot \mathcal H^1 (S_r) ^2 \geq C\cdot  \mathcal H^1 (\gamma) ^2  \geq b(r)$. Or otherwise, $\gamma$ is a simple curve connecting two points on $\partial Z$, and not intersecting $\partial Z$ in further points.
 The shorter part of $\partial Z$ between these two points has length bounded from above by $L\cdot \length _Z (\gamma )$, due to the chord-arc condition.   Moreover, if $\delta $ has been chosen small enough, this shorter part of $\partial Z$ does not contain the point $q$.
  Therefore, the Jordan curve
 $\hat \gamma $ consisting of $\gamma$ and the piece of $\partial Z$
 we have found, has length bounded from above by $(1+L)\cdot \ell _Z (\gamma)$. Moreover, the closure of  the Jordan domain of $\hat \gamma$ contains $z_0$, hence the whole ball $B(z_0,r)$ by construction.
 Now we apply the isoperimetric inequality  \tref{prop:Z-isop}  to the curve $\hat \gamma$ to deduce that $C\cdot \mathcal H^1(\hat \gamma )^2  \geq b(r)$.  Since $  (1+L)\cdot \mathcal H^1 (S_r)  \geq   (1+L)\cdot \mathcal H^1 (\gamma)  \geq \ell _Z (\hat \gamma ) $  we obtain the desired inequality \eqref{last} with
 $k= C\cdot (1+L)^2$.
\end{proof}

\subsection{Different choices of the family of curves}\label{subsec:changes-different-curve-families}

As already mentioned in Section~\ref{subsec:subtle}, all of our results concerning the space $Z$ remain valid if in the definition of the pseudo-metric $d_u$ one uses rectifiable, piecewise biLipschitz, or piecewise smooth curves
instead of continuous curves. Taking into account the following observations, the proofs remain literally the same as the ones given above. Firstly, Corollary~\ref{twomaps} remains valid if  \eqref{eq:twomaps} is only assumed for  $p$-almost all piecewise smooth  curves.  Secondly, the curves constructed  in Lemma~\ref{mainboundary}  and Corollary~\ref{shortconnection} are piecewise smooth. Thirdly, in Lemma~\ref{prop:first+}, the last equation remains valid   for all curves $\gamma$ in the chosen family.

\section{The absolute minimal filling} \label{sec:absolute}
 \subsection{The proof of  \tref{fillingareathm}}
Let us fix a quasi-convex definition of area $\mu$ and a biLipschitz circle $\Gamma$.
We consider an isometric embedding of $\Gamma$ into its  injective  hull $i:\Gamma\to Y$.  Concerning
the definition and properties of injective metric spaces and injective hulls we refer e.g. to \cite{Lan13},  \cite{Isb64}.
 Recall that $Y$  satisfies the  $(C,\infty)$-isoperimetric inequality for any  $\mu$ and $C=\frac 1 {2\pi}$,  \cite{LW}, Lemma 10.3.

As  proved in \cite{LW}, Corollary 10.4, the Sobolev filling area
$$m_{\mu,Sob} (\Gamma):=\inf \{Area _{\mu} (u) \; : \; G \; \text{complete}, \; \Gamma \subset G, \;
  u\in \Lambda (\Gamma,G)\}$$
  is realized by   a solution $u$ of the Plateau problem for $(\Gamma,Y)$ with respect to
$\mu$.  Denoting by $\gamma_0:S^1\to \Gamma$ a parametrization proportional to arclength, we deduce from \lref{lem:independent}, cf. also \cite{LW-harmonic}, Corollary 3.3:
\begin{equation} \label{eq:liptr}
m_{\mu,Sob} (\Gamma)=\inf \{Area _{\mu} (u) \; : \; u\in \Lambda (\Gamma, Y) \; , \; \trace (u) =\gamma _0\}.
\end{equation}

In order to compare the  Sobolev filling area $m_{\mu, Sob} (\Gamma )$  with Gromov's restricted filling area $m_{\mu} (\Gamma )$, we recall the following result of S. Ivanov proven in   \cite{Iva08}. The restricted filling area $m_{\mu} (\Gamma )$ is the infimum over all $\mu$-areas of Lipschitz maps $v:\bar D\to G$ into some metric space $G$ containing $\Gamma$, such
that the restriction of $v$ to $S^1$ is  a  bi\-Lipschitz parametrization of $\Gamma$. Since any such map $v$ is a Sobolev map in $\Lambda (\Gamma, G)$, we get:
$$m_{\mu,Sob} (\Gamma) \leq m_{\mu} (\Gamma).$$
The reverse inequality is  a  direct consequence of \eqref{eq:liptr} and the following lemma,
 whose proof is essentially contained in  \cite{HKST15}, Theorem 8.2.1.

\begin{lem} \label{lem:injective}
Let $Y$ be an injective metric space  and  let $u\in W^{1,2} (D,Y)$ be  such that  trace $tr(u):S^1\to Y$ is  Lipschitz continuous.  Then for every $\epsilon >0$ there exists a  Lipschitz map $v:\bar D\to V$ with  $tr(v)= tr(u)$
and $Area_{\mu} (v) <  Area_{\mu} (u) +\epsilon$.
\end{lem}

 \begin{proof}
We denote by $B$ the ball $B(0,2)\subset \R^2$. We extend $u$ to a
map $\hat u \in W^{1,2} ( B,Y)$ by $\hat u(rz)=u(z)$ for $r>1, z\in S^1$.
By assumption, the map $\hat u$ is Lipschitz continuous on $B\setminus D$ and $\Area_\mu(\hat{u}) = \Area_\mu(u)$.

Fix $\epsilon >0$.
As in   the proof of \cite{HKST15}, Theorem 8.2.1, there exist a sufficiently large $t>0$ and   a set $E_t \subset B$  such that $\hat u:B\setminus E_t\to Y$ is $t$-Lipschitz continuous.
Moreover, the Lebesgue measure of $E_t$ is at  most $\frac \epsilon {t^2}$. Finally, by the construction in [HKST15], Theorem 8.2.1, for
sufficiently large $t>0$ the set $E_t$ is contained in the ball
$B(0,\frac{3}{2})$.  Since $Y$ is injective,  we find
some  $t$-Lipschitz extension $v:B\to Y$
 of $\hat u|_{B\setminus E_t}$.

Since $v$ and $\hat{u}$ coincide on $B\setminus E_t$ and since
$\Area_\mu(v|_{E_t})\leq t^2\cdot \frac{\epsilon}{t^2}= \epsilon$ it
follows that $\Area_\mu(v) - \Area_\mu(u)\leq \epsilon$. Moreover, after
rescaling the ball $B$ so that $v$ is defined on $D$ we clearly have
$\trace v = \trace u$.
\end{proof}

\begin{rmrk}
 The statement of \lref{lem:injective} (and its proof, up to minor modifications)  remains valid if $Y$ is  $1$-Lipschitz connected up to some scale, see \cite{LW15-asymptotic}, Proposition 3.1.
\end{rmrk}

 The remainder of \tref{fillingareathm} is a consequence of the previous results.  Indeed, consider our $\mu$-minimal map  $u\in \Lambda (\Gamma, Y)$
 with $m_{\mu,Sob} (\Gamma ) =\Area_{\mu} (u)$.  Consider its unique continuous extension $u:\bar D\to Y$ and the intrinsic metric space  $Z$ defined via the pseudo-distance $d_u$ on  $\bar D$.  Consider the corresponding projection $P:\bar D\to Z$.
The space $Z$ is compact, geodesic  and homeomorphic to $\bar D$ by \lref{lem:nonsense-P}. The remaining statements  of  \tref{fillingareathm} are direct consequences of
\tref{thmA+}, \tref{thmB+}, \tref{thmC+} and \tref{arc-chord+}.

\subsection{Absolute minimizers}
Following \cite{Iva08}, \cite{minfill} we say that a geodesic metric space $M$ biLipschitz homeomorphic to $\bar D$ is an \emph{absolute minimal filling (of its boundary with respect to the definition of area $\mu$)}
if $\mu (M) = m_{\mu} (\partial M)$. Due to \tref{fillingareathm}, this implies  $m_{\mu,Sob} (\partial M )=
\mu (M)$.
The following  classes of spaces are important examples of absolute minimal fillings.
\begin{example}
Let $V$ be a two-dimensional normed space  and let  $M$ be any closed subset of $V$ biLipschitz homeomorphic to $\bar D$. Then $M$ with its induced intrinsic metric is an absolute minimal filling with respect to any quasi-convex $\mu$. Indeed,  consider the injective hull $W$ of $V$. Then $W$ is a
Banach space which contains $V$ as a linear subspace, \cite{Isb64}. By the definition of quasi-convexity, $\mu (M)$ equals the infimum of $\mu$-areas of Lipschitz discs in $W$ which have $\partial M$ as their boundaries. Due to the injectivity
of $W$ this implies  $\mu (M)=m_{\mu} (\partial M)$.
\end{example}

\begin{example}
Let $M$ be a two-dimensional smooth Finsler manifold homeomorphic to $\bar D$.  If all local geodesics in $M$ are globally minimizing then $M$ is an absolute minimal filling with respect to
the Holmes-Thompson definition of area $\mu ^{ht}$, see \cite{minfill}.
\end{example}

For any biLipschitz circle $\Gamma$,
\tref{fillingareathm} provides a generalized minimal filling of $\Gamma$, which may be slightly less regular than a biLipschitz disc.
We hope to investigate further properties of such generalized  minimal fillings in a continuation of this paper, cf. Question \ref{lastquest} below.

\section{Examples and Questions} \label{sec:exquest}
The first example is well known, see \cite{MR02}.
\begin{example} \label{alpha}
Let $X$ be the Euclidean cone over a circle $\Gamma$ of length $2\pi \alpha $ with $0<\alpha \leq 1$.
The space has property (ET) and admits a $(C,l_0)$-isoperimetric  inequality
for any $l_0$ and $C=\frac 1 {4\pi \alpha}$,
for any definition of area $\mu$. The unique solution of the Plateau problem for the curve $\Gamma$ is
$\alpha$-H\"older continuous, but  not $\beta$-H\"older continuous for any $\beta >\alpha$.
The arising space $Z$ coincides with $X$ and $P$ coincides with $u$. The balls of radius $r<1$ around the origin have area $\alpha \cdot \pi r^2 =\frac 1 {4 C} r^2$.
\end{example}

In the last example, the map $P$ is  not Lipschitz continuous for $\alpha <1$, but $Z$ is still  biLipschitz equivalent to $\bar D$.  Moreover, for $\alpha \to 1$ the isoperimetric constant tends to the critical value $\frac 1 {4\pi} $.
Nevertheless,
the solution of the Plateau problem need not  be  Lipschitz continuous.
Therefore, the answer to the next question  must involve  very fine invariants of spaces.

\begin{quest}
What conditions, apart from upper curvature bounds,  imply that solutions of the Plateau problems
are Lipschitz? Under which conditions is  $Z$ biLipschitz homemomorphic to a  disc?
\end{quest}

Natural examples to
study in connection with the last question might be Finsler manifolds and Riemannian manifolds with non-smooth metrics.
The next example mentioned in the introduction is a typical counterexample for many results  valid in the smooth case.

\begin{example} \label{disc-collapse}
Choose a compact metric ball  $T\subset D$  and consider the quotient metric space $X=\bar D/T$ with the quotient metric,  see e.g.~\cite[Definition 3.1.12]{BBI01}. Then $X$ is a geodesic space, homeomorphic to a disc, and $X$ is flat outside a single thick point, the image of $T$.
It follows from the Euclidean isoperimetric inequality that $X$ admits a $(C,\infty)$-isoperimetric inequality with optimal constant $C=\frac 1 {2\pi}$.
The "worst" curves, enclosing the maximal area, are projections  of tiny  circles which meet the boundary of $T$ orthogonally.

The canonical projection $u\colon\bar D\to X$ is a conformal solution of the  Plateau problem.
The metric space $Z$ coincides with $X$  and $u$ coincides with $P$.    The minimal disc  $u$ has the set $T$   as the set of "branch points".  Moreover, small balls around non-thick points in $Z=X$  have quadratic area growth  and small balls around the thick point have  a linear area  growth.  Thus, the Hausdorff area is not a doubling measure on $Z$. In particular, $Z$ is not biLipschitz  homeomorphic to a  disc.
\end{example}

In view of the nature of this example
it seems possible  that the
isoperimetric constant $\frac 1 {2 \pi}$
 is the critical value for  the constructions of this type. Note that this value
$C = \frac 1
{  2  \pi} $  is also very interesting in view of the absolute Plateau problem.

\begin{quest} Can the set of branch points of a solution of the Plateau problem be large
if the isoperimetric constant $C$ is smaller than $\frac 1 {2\pi}$?
 Can  the map $P$ be  non-injective  in this case?
\end{quest}

The fibers of $P$, which are a priori allowed by the
statement  that $P$  is a uniform limit of homeomorphisms, can be arbitrary cell-like sets, for instance any continuous simple arc.
We do not know if such general fibers can indeed
occur.

\begin{quest}
Can fibers of the canonical map $P$ be non-contractible? Are such fibers always Lipschitz retracts?
\end{quest}

 The following question is closely related to the previous one. Due to \tref{prop:first}
the question has an affirmative answer if we find controlled approximations of any curve in $Z$ by $P$-images of curves in $\bar D$.
\begin{quest}
Does the  map $\bar u :Z\to X$ preserve the lengths of all curves in $Z$?
\end{quest}

  Finally, we do not know to what extent the conclusions about absolute minimizers
are optimal.

\begin{quest} \label{lastquest}
Are solutions of the absolute Plateau problem Lip\-schitz continuous?   What can be said about their geometry?
\end{quest}

\def\cprime{$'$} \def\cprime{$'$}

\end{document}